%% file: main_article.tex
\documentclass[hidelinks,onefignum,onetabnum]{siamonline250211}

\newsiamremark{assumption}{Assumption}
\usepackage{MnSymbol}
\usepackage{multirow}
\usepackage{subcaption}
\usepackage{tablefootnote}
\usepackage{array}
\definecolor{Red}{rgb}{1,0.25,0.25}
\definecolor{Green}{rgb}{0.25,0.75,0.25}
\definecolor{Blue}{rgb}{0.1,0.5,1}

\newcommand{\adj}[1]{\operatorname{adj}(#1)}

\usepackage{amsmath}
\usepackage{amsfonts}
\usepackage{graphicx}
\usepackage{enumitem}

\usepackage{comment}
\usepackage{todonotes}

\usepackage{comment}
\usepackage[most]{tcolorbox} 

\input{ex_shared}

\ifpdf
\hypersetup{
  pdftitle={Blow-Up of a Tumour Model}
  pdfauthor={T. E. F. Lapuz and M. Wechselberger}
}
\fi

\begin{document}
 \title{Escape, Oscillatory Relapse or Elimination:\\ A Geometric Blow-up Analysis of a Tumour Model with Singular Bifurcations}
\author{Timothy Earl Figueroa Lapuz\thanks{School of Mathematics and Statistics, The University of Sydney, Camperdown NSW 2006, Australia} \and Martin Wechselberger\footnotemark[1]}
  
\maketitle

\begin{abstract}

The five-compartment tumour-immune model introduced by Kuznetsov et al.~captures critical clinical phenomena, including tumour dormancy, immune evasion (`sneaking through'), and immunostimulation. Recent studies by Osojnik et al.~highlighted the model's excitable and oscillatory dynamics, providing a mechanism for tumour relapse. In this paper, we recast these findings through the lens of \textit{geometric singular perturbation theory (GSPT)} and the \textit{blow-up method}, establishing a formal mathematical foundation for these phenomena. 

Because the model features a multiple-time-scale structure with dynamics separated by up to three orders of magnitude, standard Tikhonov-Fenichel theory is insufficient. To resolve this, we employ the \textit{parametrisation method} to systematically compute higher-order approximations of the slow vector fields. We introduce a methodological novelty by demonstrating that the parametrisation method is indispensable within the geometric blow-up analysis. Specifically, we show that higher-order correction terms to the local centre manifolds are required to match the vector field on the corresponding slow manifold. 

Using these combined tools, we rigorously prove the existence of stable relaxation oscillations and transient large excursions. We uncover a rich collection of singular bifurcations, including a saddle-node on invariant circle (SNIC) and two distinct singular Andronov-Hopf (sAH) bifurcations that drive complete and incomplete canard explosions. At the highly degenerate origin, our blow-up analysis reveals the spatial collision of a singular transcritical bifurcation and a nilpotent pseudo-singularity. We conjecture that this collision forms a novel \textit{singular transcritical Bogdanov-Takens (stBT)} organising centre. Ultimately, these geometric findings yield a precise bifurcation diagram in the parameter space of effector cell supply and death rates, partitioning clinical outcomes into tumour escape, oscillatory relapse, and tumour dormancy or elimination.
\end{abstract}



\section{Introduction}
\label{sec:intro}
Biological phenomena, such as the interaction between proliferating tumour cells and the immune system, frequently exhibit dynamics across vastly disparate time scales. In this paper, we revisit a five-compartment tumour model originally proposed by Kuznetsov et al.~\cite{kuznetsov}, and recently analysed by Osojnik et al.~\cite{osojnik}. The model captures critical clinical phenomena, including tumour dormancy, `sneaking through' (where small tumours escape immune control), and immunostimulation. Osojnik et al.~demonstrated that these dynamics manifest mathematically as excitable large excursions and relaxation oscillations, providing a dynamical mechanism for biological relapse, where high tumour levels return after extended periods of remission. 

Such multiple-time-scale phenomena are naturally suited for geometric singular perturbation theory (GSPT) \cite{tikhonov,fenichel, jones, kuehn2015, wechselberger2020} and the blow-up method \cite{dumortierroussarie1996, krupaszmolyan2001fold, szmolyanwechselberger}. The combination of GSPT and blow-up has been successfully employed to prove the existence of relaxation oscillations in various biological models \cite{kosiukszmolyan, kosiukszmolyan2, kosiukszmolyan3, process}. {However, the Kuznetsov et al. model presents a unique mathematical challenge. Unlike conventional singularly perturbed models, this model contains a multiple-time-scale structure, with scales ranging from $\mathcal{O}(1)$ (fast) to $\mathcal{O}(\varepsilon^3)$ (infra-infra-slow)}. The basic GSPT toolbox struggles to systematically capture dynamics across these disparate scales. To resolve this, we use the \textit{parametrisation method} {\cite{coulletspiegel,cabre20031, cabre20032, cabre2005,param}}, recently adapted for GSPT by Lizarraga et al.~\cite{multiple}, which recursively computes higher-order approximations of normally hyperbolic slow manifolds and their corresponding slow vector fields.

This multiple-time-scale behaviour reveals a rich collection of singular bifurcations that organise the global dynamics. {Within a \textit{nested} infra-slow critical manifold contained in a larger slow manifold, we identify a folded singularity giving rise to a singular Andronov-Hopf (sAH) bifurcation, alongside a saddle-node on invariant circle (SNIC) bifurcation. These bifurcations partition the parameter space into distinct regimes of excitable dynamics (associated with tumour escape) and oscillatory relapse. More strikingly, at the origin, normal hyperbolicity is severely lost. Our analysis uncovers a nilpotent pseudo-singularity, from which a second sAH arises, partitioning the parameter space between oscillatory relapse and tumour dormancy or elimination. Furthermore, we uncover a singular transcritical (sTC) bifurcation. Its spatial with the nilpotent pseudo-singularity is conjectured to form a novel {\emph{singular transcritical Bogdanov-Takens (stBT)} bifurcation.}}

The main contributions of this paper are threefold:
\begin{itemize}
   \item {We provide a formal geometric proof for the existence of stable relaxation oscillations and excitable large excursions, establishing rigorous mathematical mechanisms for tumour escape, oscillatory relapse, and dormancy.}
    \item {We derive analytic boundaries in the parameter space of effector cell supply and death rates, delineating regions for distinct patient outcomes. Furthermore, we describe how the sAH bifurcations generate complete and incomplete canard explosions.}
    \item {We extend the utility of the parametrisation method beyond classical Tikhonov-Fenichel theory by demonstrating how it facilitates the successful application of the blow-up method. Specifically, we show that higher-order correction terms to the centre manifolds in both an entry chart and an exit chart are required. This technical necessity ensures that the local vector fields on the centre manifolds perfectly match the slow vector field on the corresponding slow manifold, a matching condition that strictly relies on the parametrisation method. To our knowledge, this requirement has not been explicitly addressed in previous applications.}
\end{itemize}

The remainder of the paper is structured as follows. Section \ref{model} introduces the model, establishes the non-standard singular perturbation framework, and presents a numerical bifurcation analysis. Section \ref{sec:singular} applies the parametrisation method to extract the nested slow flows on the slow manifolds. Section \ref{sec:bifs11} analyses the singular bifurcations on these manifolds, including a saddle-node (SN) bifurcation that becomes a SNIC bifurcation, and an sAH bifurcation that triggers an associated canard explosion. Section \ref{sec:blowup_tumour} deploys the geometric blow-up method at the conjugate $C$-axis and the origin to prove the global return mechanism, culminating in our main existence theorem for stable relaxation oscillations. Finally, Section \ref{sec:degenerate bifurcations} investigates the highly degenerate bifurcations near the origin, revealing an sAH and an sTC, as well as the conjectured stBT organising centre, before we conclude in Section \ref{conclusion}.

\section{The model}
\label{model}
Consider the five-compartment model proposed by Kuznetsov et al. \cite{kuznetsov} and revisited by Osojnik et al. \cite{osojnik},
\begin{align}
    \frac{d \mathcal{E}}{dT} &= s + \frac{f \mathcal{C}}{g + \mathcal{M}} - d_1 \mathcal{E} - k_1 \mathcal{E} \mathcal{M} + (K - k_3) \mathcal{C}, \,\, \quad \quad \quad \textcolor{gray}{\frac{d\mathcal{E}^*}{dT} = k_3 \mathcal{C} - d_2 \mathcal{E}^*} \nonumber \\
    \frac{d \mathcal{M}}{dT} &= a \mathcal{M} (1 - b \mathcal{M}) - k_1 \mathcal{E} \mathcal{M} + (K-k_{2}) \mathcal{C}, \quad \quad \quad \quad \textcolor{gray}{\frac{d\mathcal{M}^*}{dT} = k_2 \mathcal{C} - d_3 \mathcal{M}^*}  \nonumber \\
    \frac{d\mathcal{C}}{dT} &= k_1 \mathcal{E} \mathcal{M} - K \mathcal{C} \label{3D} 
\end{align}
where $\mathcal{E}, \mathcal{M}, \mathcal{C}, \mathcal{E}^*, \mathcal{M}^*$ are local concentrations of effector cells, tumour cells, intermediate conjugate, inactivated effector cells, and lethally hit tumour cells, respectively (see Figure~\ref{fig:tumour_diagram} for a schematic diagram). Following \cite{osojnik}, we define $K := k_{-1} + k_2 + k_3$. The initial conditions (ICs) for system \eqref{3D} are $\mathcal{E}(0) = \mathcal{E}_0, \, \mathcal{M}(0) = \mathcal{M}_0$ with all other compartments initially zero. Parameter values are given in Table \ref{parameter_values}.

\begin{figure}[ht]
\centering
 \includegraphics[width=0.475\linewidth]{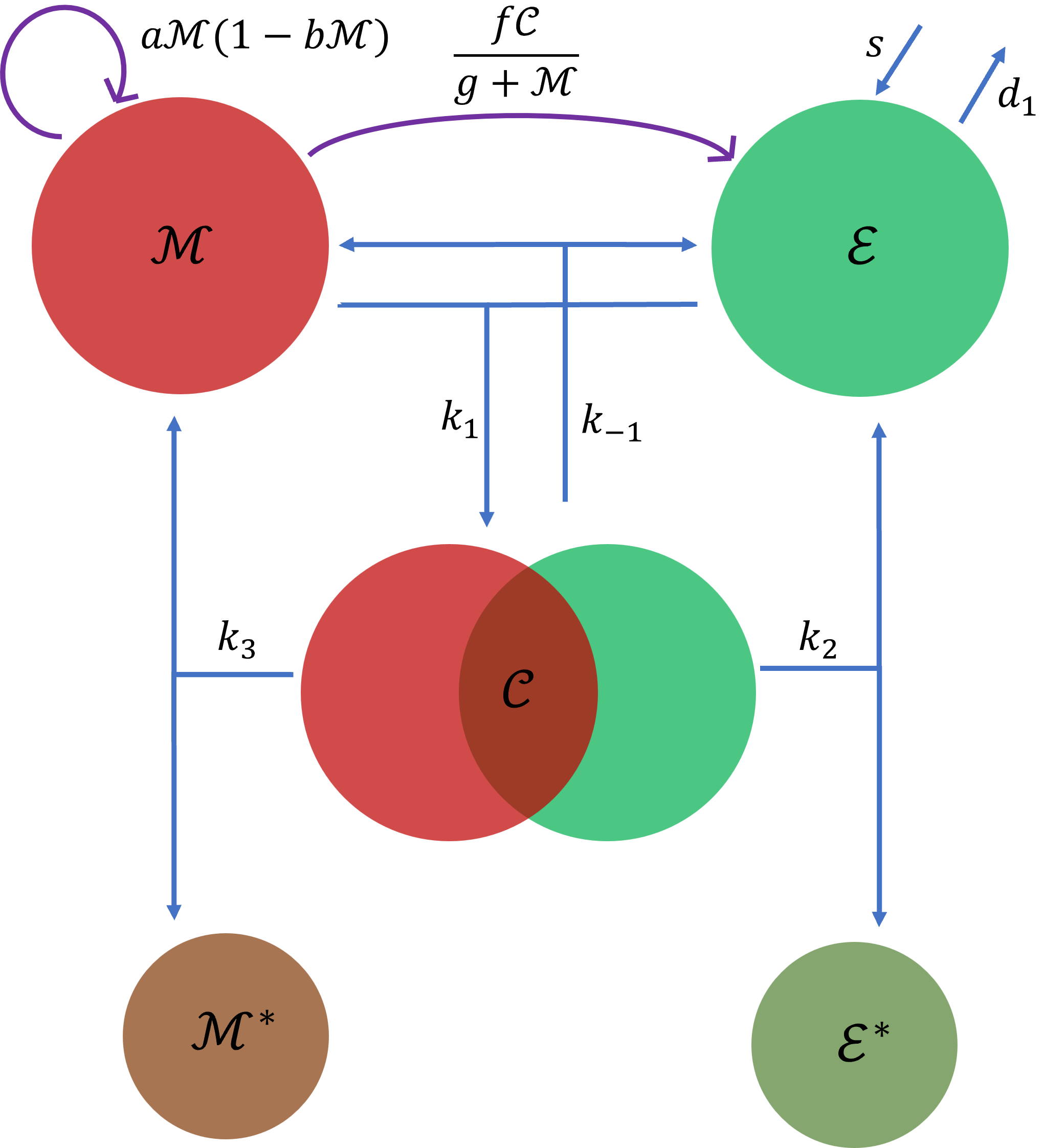}
  \caption{Schematic diagram of the tumour cell $\mathcal{M}$ and effector cell $\mathcal{E}$ dynamics as described in the main text. Blue lines denote reactions that are mass action while purple lines are beyond mass action.}
\label{fig:tumour_diagram}
\end{figure}

\begin{table}[ht]
\centering
\begin{tabular}{ |c|c| }
 \hline
Parameter & Value \\ \hline
$a$ & 0.18 day$^{-1}$ \\
$b$ & $2.0\times 10^{-9}$ cell concentration$^{-1}$ \\ 
$d_1$ & $0.0412$ day$^{-1}$ \\
$f$ & $1.245 \times 10^5$ cell concentration day$^{-1}$ \\
$g$ & $2.019\times 10^7$ cell concentration \\
$k_2$ & $0.1101$ day$^{-1}$ \\
$k_3$ & $3.422 \times 10^{-4}$ day$^{-1}$ \\
$s$ & $1.3\times10^4$ cell concentration day$^{-1}$ \\
$k_1$ & $10^{-4}$ cell concentration$^{-1}$ day$^{-1}$ \\
$k_{-1}$ & $99.900$ day$^{-1}$ \\
 \hline
\end{tabular}
\caption{Dimensional parameter values, adapted from \cite{kuznetsov,osojnik} based on theoretical estimates and mouse spleen tumour growth data \cite{siuetal}. {Since exact values for $k_1$ and $k_{-1}$ were not specified in the original literature, the values chosen here were selected to satisfy the three constraints established by \cite{osojnik}: (i) $\frac{k_{-1} + k_2 + k_3}{k_1} \approx 10^6$, (ii) $k_1 \times 10^{6} = \mathcal{O}(k_{-1})$, and (iii) $k_2, k_3, s\times 10^{-6}, \frac{fk_1 }{k_{-1} + k_2 + k_3}, d_1, a \ll k_{-1}$.}} \label{parameter_values}
\end{table}

Since the dynamics of $\mathcal{E}^*$ and $\mathcal{M}^*$ decouple from the primary variables, we restrict our analysis to the 3D subsystem $(\mathcal{E}, \mathcal{M}, \mathcal{C})$. While most terms in \eqref{3D} invoke standard mass action, the tumour-stimulated effector supply $\frac{f \mathcal{C}}{g+ \mathcal{M}}$ and the logistic tumour growth $a \mathcal{M}(1-b\mathcal{M})$ are phenomenological choices by Kuznetsov et al. \cite{kuznetsov}. Mathematically, \eqref{3D} may be viewed as a reduced chemical reaction network \cite{lapuzwechselberger}. Numerical simulations presented in Figure \ref{time_trace} suggest that system \eqref{3D} undergoes multiple timescale dynamics; for these simulations, we set $s=5\times 10^4$ to ensure the system is in the oscillatory regime.

\begin{figure}[ht]
\centering
\begin{subfigure}{0.45\textwidth}
\centering
 \includegraphics[width=1\linewidth]{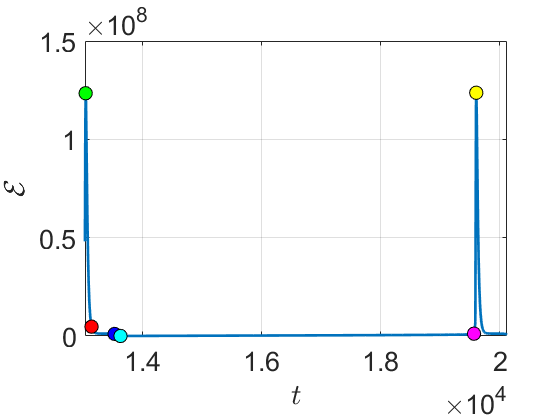}
  \caption{}
\end{subfigure}
\begin{subfigure}{0.45\textwidth}
\centering
  \includegraphics[width=1\linewidth]{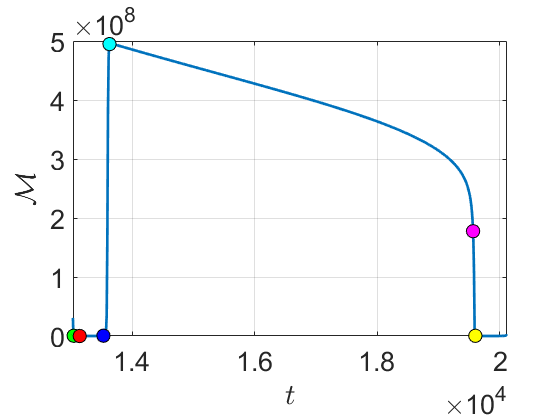}
  \caption{}
\end{subfigure}
\begin{subfigure}{0.45\textwidth}
\centering
  \includegraphics[width=1\linewidth]{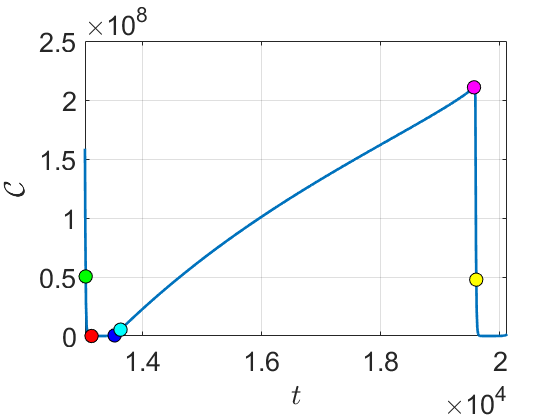}
  \caption{}
\end{subfigure}
\begin{subfigure}{0.45\textwidth}
\centering
  \includegraphics[width=1\linewidth]{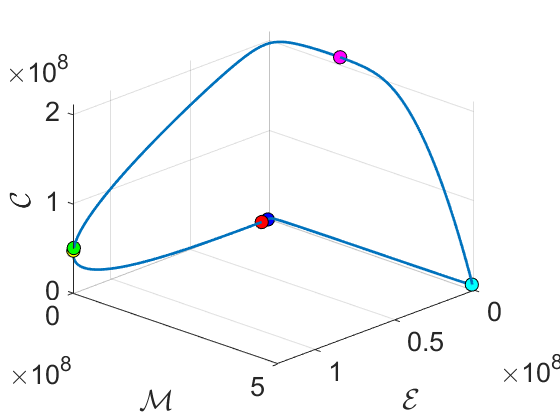}
  \caption{}
\end{subfigure}
\caption{(a)-(c) Time traces of system \eqref{3D} for $\mathcal{E},\mathcal{M},\mathcal{C}$. (d) Phase portrait. ICs: $(\mathcal{E}_0 ,\mathcal{M}_0, \mathcal{C}_0) = (10^6,10^2,0)$ with $s = 5\times 10^4$ to induce oscillations. During a single period of $\mathcal{O}(10^4)$ time units, the system undergoes a rapid drop in $\mathcal{E}$ leading to a fast escape of $\mathcal{M}$. The trajectory then tracks a slow manifold as $\mathcal{C}$ gradually increases (conjugate formation), terminating in a fast drop in $\mathcal{M}$ and $\mathcal{C}$ coupled with a rapid influx of $\mathcal{E}$.}
\label{time_trace}
\end{figure}

\subsection{Nondimensionalisation and singular perturbation setup}
To prepare system \eqref{3D} for multiple time scale analysis, we nondimensionalise the coordinates $(\mathcal{E},\mathcal{M},\mathcal{C})$ and time $T$ (derivations and reference scales are detailed in Section SM1 of Supplementary Material I). To remove rational terms, we apply a topological equivalence time desingularisation $dt = (M+\delta)^{-1}d\tau$. 
We identify the primary singular perturbation parameter as $\varepsilon := \frac{Kb}{k_1} \ll 1$. The remaining dimensionless parameters are grouped by their asymptotic scaling with $\varepsilon$:
\begin{align}
    \delta = bg, \quad \tilde{\alpha}_{1} &= \varepsilon^3 \alpha_{1}, \quad \tilde{\alpha}_{2} = \varepsilon^3 \alpha_{2}, \quad \tilde{\beta}_{1} = \varepsilon^2 \beta_{1}, \nonumber \\ 
    \tilde{\beta}_{2} &= \varepsilon^2 \beta_{2}, \quad \tilde{\gamma}_1 = \varepsilon^3 \gamma_1, \quad \tilde{\gamma}_2 = \varepsilon^2 \gamma_2, \label{nondim_parameters2}
\end{align}
yielding the $\mathcal{O}(1)$ parameter values summarised in Table \ref{parameter_values_new}.

\begin{table}[ht]
\centering
\begin{tabular}{ |c|c|c|c| } 
 \hline
Parameter & Value & Parameter & Value  \\ \hline
$\delta$ & $4.038 \times 10^{-2}$ & ${\beta}_2$ & $0.9$\\
$\varepsilon $ & $2 \times 10^{-3}$ & ${\gamma}_1$ & $0.8555$ \\
${\beta}_1$ & $0.2060$ & ${\gamma}_2$ & $0.5505$ \\
${\alpha}_1$ & $0.065$ & ${\alpha}_2$ & $0.6225$\\
\hline
\end{tabular}
\caption{Dimensionless parameter values at $\mathcal{O}(1)$, scaled by $\varepsilon$.} \label{parameter_values_new}
\end{table}

\begin{remark} \label{remark:osojnik_differences} 
{The approach in \cite{osojnik} first selects a small parameter for a reduction to 2D, then introduces ad-hoc small parameters for further analysis of excitability. Here, $\varepsilon := Kb/k_1$ is defined from the dimensional kinetics and all reductions made are based on powers of $\varepsilon$.}
\end{remark}

We thus cast the system into a non-standard singular perturbation problem in $\varepsilon$:
\begin{align}
\begin{pmatrix}
E' \\ M' \\ C'
\end{pmatrix} &= N_0 f_0 +  \varepsilon F_1 + \varepsilon^2 F_2 + \varepsilon^3 F_3 \nonumber \\
&= \begin{pmatrix}
    -1 \\ -1 \\ 1
\end{pmatrix} EM(M + \delta ) + \varepsilon \begin{pmatrix} 1 \\ 1 \\ -1
\end{pmatrix} (M+\delta) C \nonumber \\ &+ \varepsilon^2 \begin{pmatrix}
    -\delta \beta_1 E - \beta_1 E M \\ -\delta \gamma_2 C + \beta_2 M^2 (1-M) - \gamma_2 M C +  \delta \beta_2  M (1-M)\\ 0
\end{pmatrix} \nonumber \\
&+ \varepsilon^3 \begin{pmatrix}
    \alpha_2 C + \alpha_1 M + \alpha_1 \delta - \gamma_1 M C - \delta \gamma_1 C
    \\ 0 \\ 0
\end{pmatrix},
\label{full_single_epsilon}
\end{align}
where prime denotes the derivative with respect to the desingularised time $t$. {The corresponding layer problem possesses a transverse intersection of invariant planes at the conjugate cell $C$-axis, signifying a loss of normal hyperbolicity. Since Tikhonov-Fenichel theory \cite{tikhonov,fenichel,jones,kuehn2015,wechselberger2020} does not apply in such cases, we will employ the blow-up method in order to resolve this loss in the subsequent sections.}

\subsection{Numerical bifurcation diagram}
The parameter $\alpha_1 \propto s$ represents the nondimensional supply rate of effector cells arriving at the tumour microenvironment, making it a critical control parameter that varies between patients and during immunotherapy \cite{osojnik}. Figure \ref{1Dbif} shows the bifurcation diagram under the variation of $\alpha_1$.
\begin{figure}[ht]
\begin{subfigure}{0.475\textwidth}
\centering
  \includegraphics[width=1\linewidth]{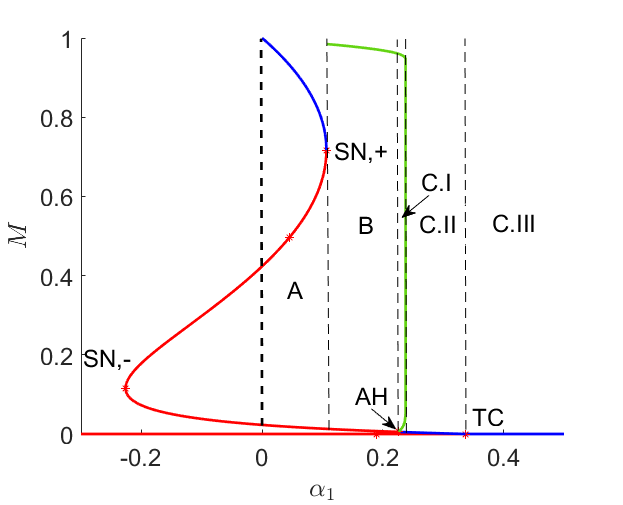}
  \caption{}
\end{subfigure}
\begin{subfigure}{0.475\textwidth}
\centering
  \includegraphics[width=1\linewidth]{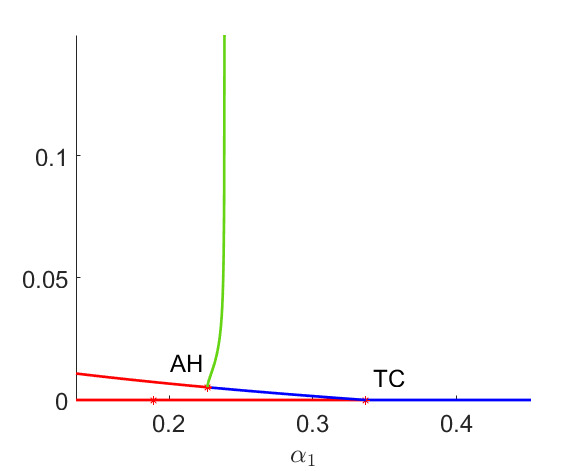}
  \caption{}
\end{subfigure}
\caption{Bifurcation diagram of system \eqref{full_single_epsilon} under the variation of $\alpha_1$. Other parameters are in Table \ref{parameter_values_new}, except we have relaxed $\varepsilon = 0.01$. The bifurcations shown are two saddle-node bifurcations (labelled SN), one subcritical Andronov-Hopf bifurcation (labelled AH), and one transcritical bifurcation (labelled TC); these were numerically obtained from MatCont \cite{matcont}. 
Continuation of the AH limit cycles suggests there is a saddle-node of periodic orbits (SNPO) on the boundary of C.I and C.II. COCO \cite{coco} was used for the limit cycle continuation. 
(a) Partitions the bifurcation diagram into regimes defined in the main text. (b) A close-up near the AH and TC bifurcations.}
\label{1Dbif}
\end{figure}
There are up to four equilibria: the tumour-free state $\mathcal{EQ}_{tf} = (E_{tf},0,0) = \left(\varepsilon \frac{\alpha_1}{\beta_1},0,0\right)$, the small-tumour state $\mathcal{EQ}_s = (E_s,M_s,C_s)$, and two large-tumour states $\mathcal{EQ}_{l,-}$ and $\mathcal{EQ}_{l,+}$ (where $M_{l,+} > M_{l,-}$). 

We can partition the bifurcation diagram into various regions based on numerical observations:
\begin{itemize}
    \item \textit{Region A (tumour escape):} $\mathcal{EQ}_{l,+}$ is the unique stable attractor. Initial conditions may undergo transient large-amplitude excursions before settling here.
    \item \textit{Region B (oscillatory relapse):} The large-tumour equilibria vanish via a saddle-node (SN) bifurcation. Solutions are asymptotically attracted to large-amplitude relaxation oscillations.
    \item \textit{Region C.I (sneaking through / bistability):} A subcritical Andronov-Hopf (AH) bifurcation stabilises $\mathcal{EQ}_s$ and generates an unstable limit cycle. This creates bistability: low initial tumour loads are trapped near $\mathcal{EQ}_s$, while larger loads cross the unstable cycle, and escape to the large relaxation oscillations.

    \item \textit{Region C.II \& C.III (tumour dormancy / elimination):} The large oscillations are destroyed by a saddle-node of periodic orbits (SNPO). $\mathcal{EQ}_{tf}$ and $\mathcal{EQ}_s$ exchange stability via a transcritical (TC) bifurcation. These equilibria are the global attractors depending on which equilibrium is stable; see Figure \ref{1Dbif}b.
\end{itemize}
  {The boundary between C.I and C.II provides a partitioning for undesirable and desirable patient outcomes, where (i) $0 \leq \alpha_{1} < \alpha_{1,SNPO}$, solutions approach the stable large-tumour equilibrium (tumour escape) or undergo large amplitude oscillations, and (ii) $\alpha_{1} > \alpha_{1,SNPO}$, solutions approach the small-tumour equilibrium (tumour dormancy) or tumour-free equilibrium (tumour elimination); see also \cite{osojnik}.}
    
 In the forthcoming analysis, we prove that solutions approach these attracting sets (i.e., an equilibrium or the relaxation oscillation) of system \eqref{full_single_epsilon}.

\section{Multiple time scale analysis}
\label{sec:singular}
Taking the singular limit $\varepsilon \to 0$ in \eqref{full_single_epsilon} yields the layer problem. The critical manifolds are described by ${S}_0=\{(E,M,C) \in \mathbb{R}^3 \mid f_0(E,M,C) = 0\}$ and comprises of two biologically relevant, non-negative invariant planes: the effector-free manifold ${S}_{0,E} = \{(E,M,C) \in \mathbb{R}^3_{\geq 0} \mid E=0\}$ and the tumour-free manifold ${S}_{0,M} = \{(E,M,C) \in \mathbb{R}^3_{\geq 0} \mid M=0\}$ (ignoring the non-physical manifold $M = -\delta$). The corresponding fast fibre bundle $\mathcal{W}_0$ of these invariant planes is given by constant linear fibres aligned with $\mathcal{N}_0 = \text{span}\{(-1,-1,1)^\top\}$. 

Evaluating the nontrivial eigenvalues of the layer problem along the invariant planes give:
\begin{align*}
    Df_0 N_0|_{{S}_{0,E}} =  -M(M + \delta), \quad Df_0 N_0|_{{S}_{0,M}} = - \delta E.
\end{align*}
Hence, ${S}_{0,E}$ is normally hyperbolic and attracting for all $M > 0$, and ${S}_{0,M}$ is normally hyperbolic and attracting for all $E > 0$. A loss of normal hyperbolicity occurs at their transverse intersection along the conjugate-axis, $L_0 = {S}_{0,E} \cap {S}_{0,M} = \{(0,0,C) \mid C \geq 0\}$, necessitating a blow-up analysis (see Figure \ref{3D_layer_scalingB}).

\begin{figure}[ht]
\centering
\includegraphics[width=0.45\linewidth]{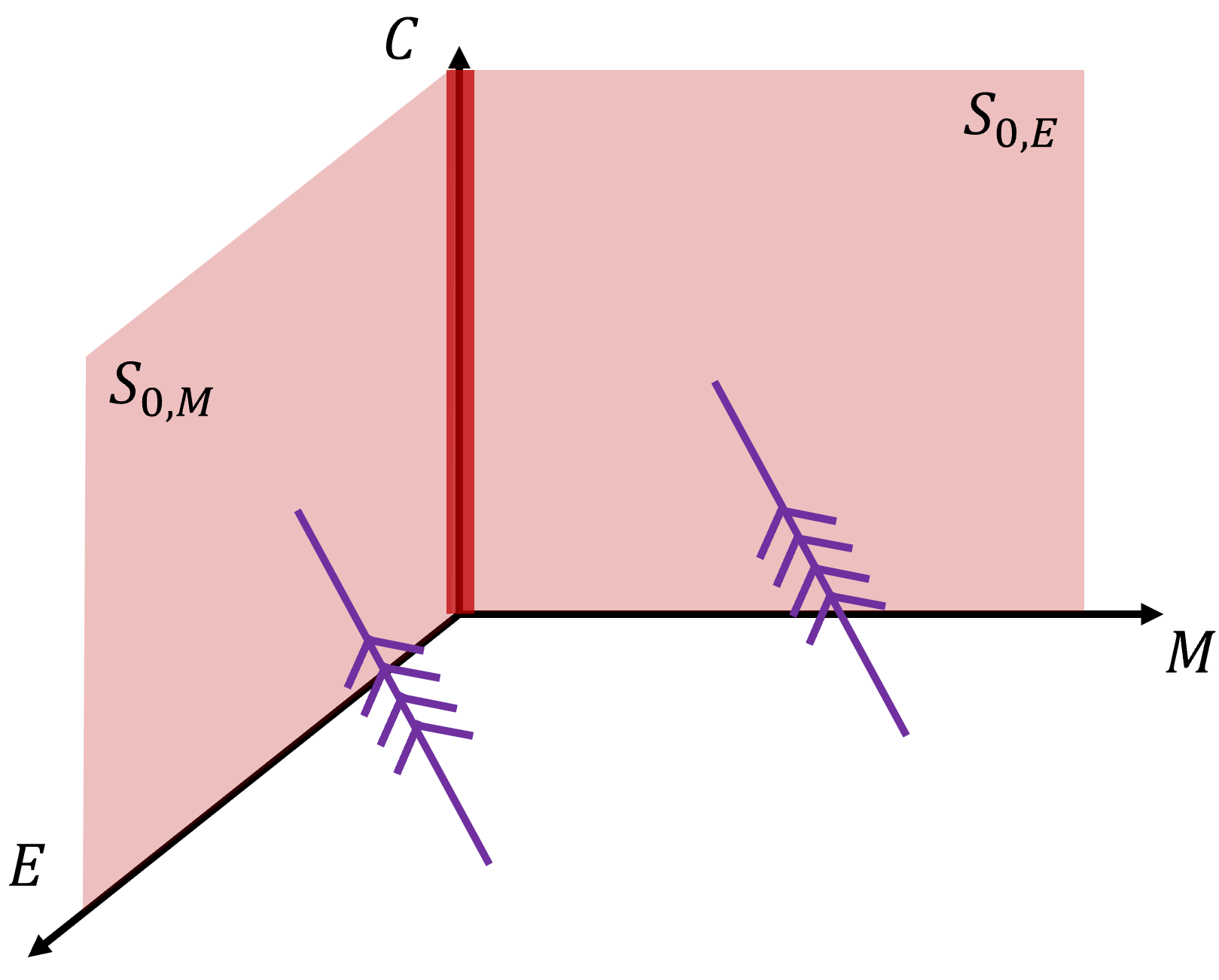}
  \caption{Sketch of the layer problem of system \eqref{full_single_epsilon}. The dark red line ($L_0$) indicates the loss of normal hyperbolicity.}
\label{3D_layer_scalingB}
\end{figure}

\subsection{The slow dynamics} 
\label{sec:slow}
To compute the slow flow on the normally hyperbolic invariant manifolds ${S}_{0,E}^\varepsilon$ and ${S}_{0,M}^\varepsilon$ (away from $L_0$), we utilise the parametrisation method for GSPT \cite{multiple,lapuzwechselberger}. This approach iteratively calculates the higher-order $\varepsilon$-corrections guaranteed by Tikhonov-Fenichel theory \cite{fenichel}. We state the resulting reduced vector fields here up to the required order; detailed derivations are deferred to Section SM2 of Supplementary Material I.

\subsubsection{Dynamics on the tumour-free slow manifold ${S}_{0,M}^\varepsilon$}
\label{zero_E}
The parametrisation method gives the following $\mathcal{O}(\varepsilon^2)$ slow vector field on ${S}_{0,M}^\varepsilon$:
\begin{align}
    \begin{pmatrix}
    E' \\ C'
    \end{pmatrix} =  \varepsilon^2 \begin{pmatrix}
        \delta(\gamma_2 C- \beta_1 E) \\-\delta \gamma_2 C
    \end{pmatrix}  + \mathcal{O}(\varepsilon^3),\label{reduction_M0}
\end{align}
where prime denotes the derivative with respect to $\tau$. The leading-order truncation has a unique, globally asymptotically stable equilibrium at the origin, corresponding to the tumour-free equilibrium $\mathcal{EQ}_{tf} = \left(\varepsilon \frac{\alpha_1}{\beta_1},0,0 \right)$ of the full system \eqref{full_single_epsilon}. The origin's corresponding eigenvalues are $\lambda_1 = - \beta_1 \delta \varepsilon^2$ and $\lambda_2 = - \delta \gamma_2 \varepsilon^2$. 

For the case $\beta_1 < \gamma_2$ (which is the case in Table \ref{parameter_values_new}), the strong stable eigendirection drives trajectories toward the weak stable subspace $W^s_{sub} = \{M=0, C=0\}$, after which they slowly funnel into the origin; see Figure~\ref{solution_M0_fig}a. {Importantly, $W^s_{sub}$ is a 1D stable invariant subspace of the full 3D system \eqref{full_single_epsilon}.}

Since all trajectories on ${S}_{0,M}^\varepsilon$ approach the origin, which lies on the degenerate line $L_0$, we must employ the blow-up method to resolve the flow here (detailed in Sections \ref{sec:blowup_tumour} and \ref{sec:degenerate bifurcations}). The case $\beta_1 > \gamma_2$ is also shown in Figure \ref{solution_M0_fig}b.

\begin{figure}[ht]
\centering
\begin{subfigure}{0.35\textwidth}
\centering
  \includegraphics[width=1\linewidth]{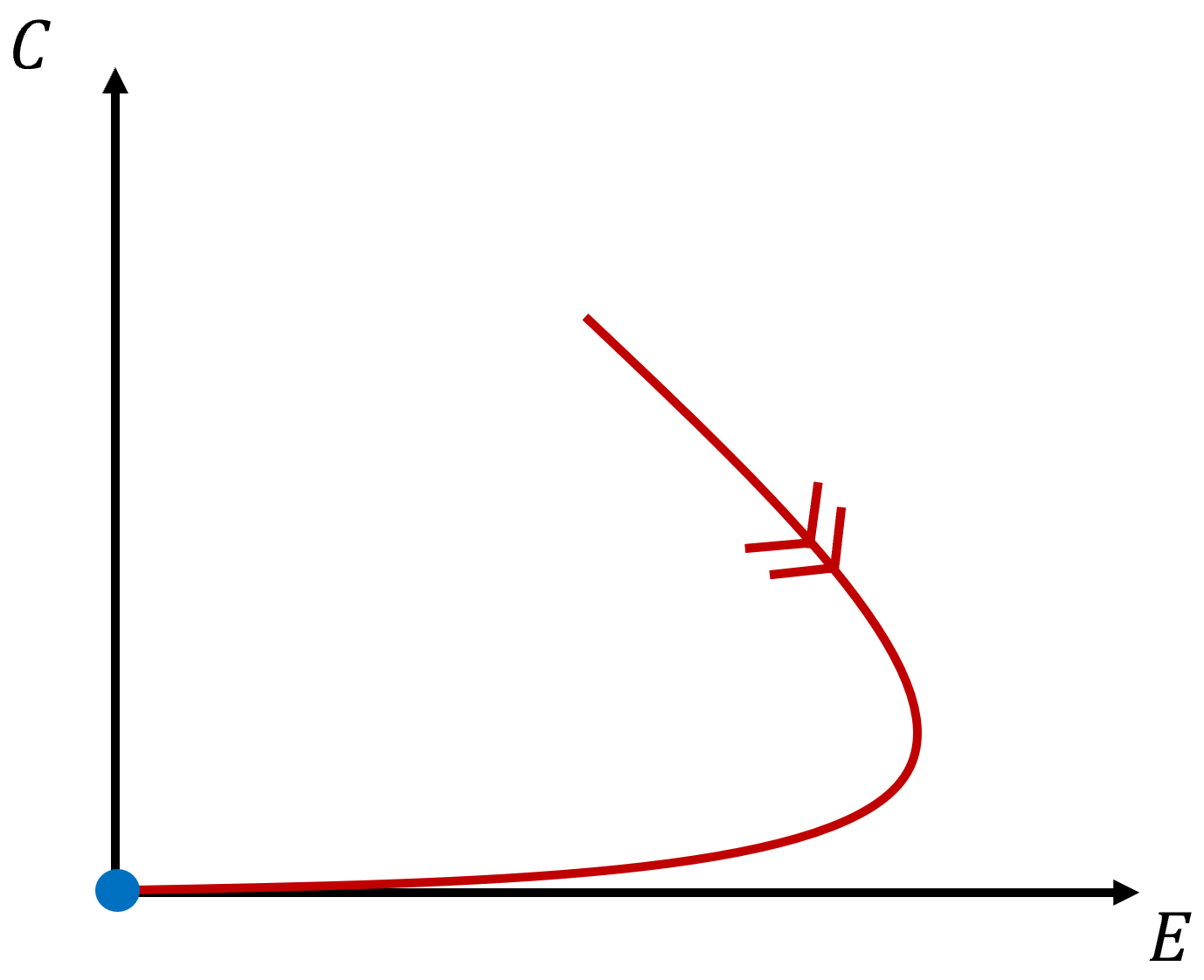}
  \caption{}
\end{subfigure}
\begin{subfigure}{0.35\textwidth}
\centering
\includegraphics[width=1\linewidth]{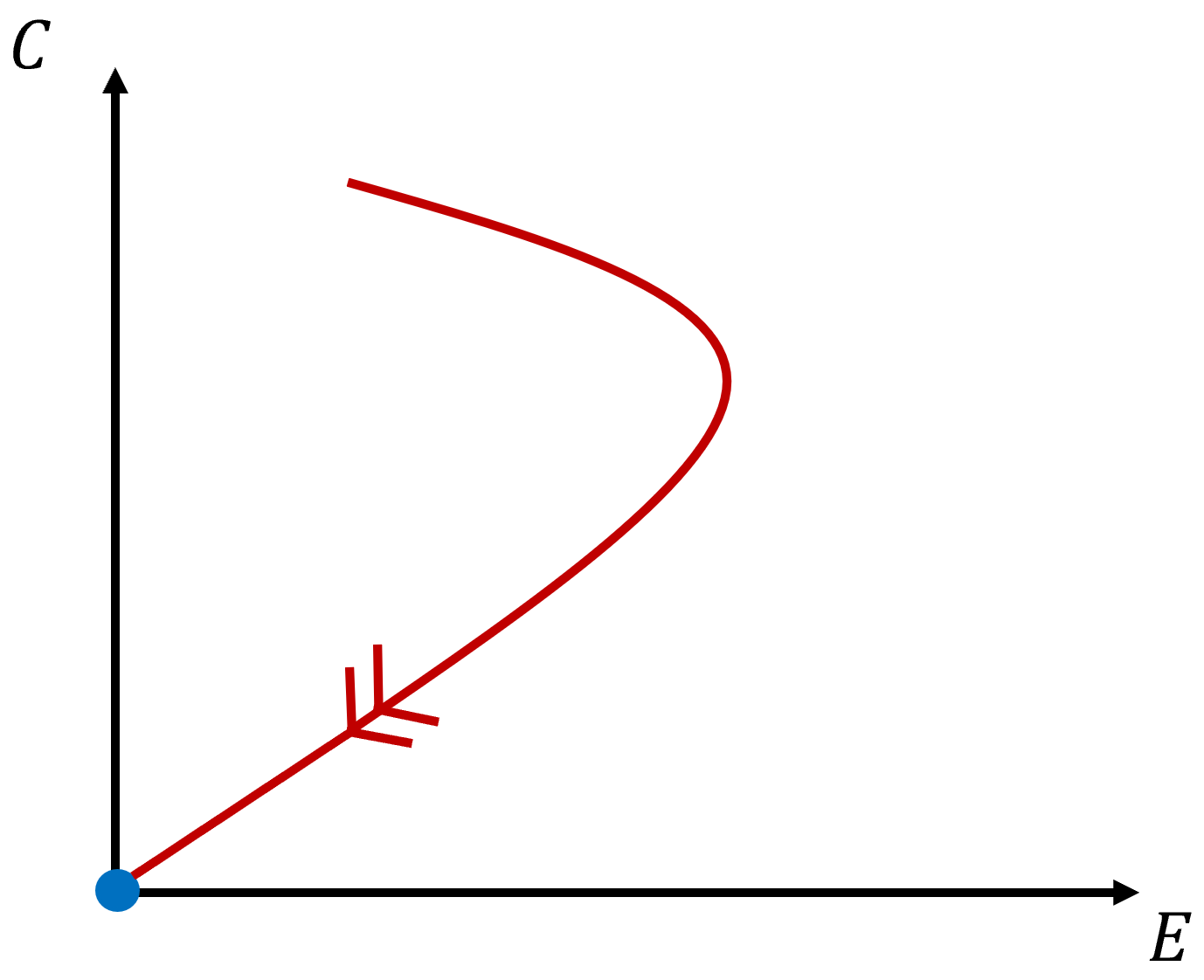}
  \caption{}
\end{subfigure}
  \caption{Sketch of a solution to system \eqref{reduction_M0}. (a) $\beta_1 < \gamma_2$. (b) $\beta_1 > \gamma_2$.}
\label{solution_M0_fig}
\end{figure}

\subsubsection{Dynamics on the effector-free slow manifold ${S}_{0,E}^\varepsilon$}
The reduced vector field on ${S}_{0,E}^\varepsilon$ takes the form:
\begin{align}
    \begin{pmatrix}
    M' \\ C'
    \end{pmatrix} &= \varepsilon^2 N_0^{(1)} f_0^{(1)} + +\varepsilon^3 R_{E,3} \nonumber \\
    &=\varepsilon^2 \begin{pmatrix}
        -1 \\ 0
    \end{pmatrix} (M+\delta)(\beta_2 M^2 -\beta_2 M + \gamma_2 C) +\varepsilon^3 R_{E,3}. \label{reduced_again2}
\end{align}
Importantly, system \eqref{reduced_again2} is \emph{itself} singularly perturbed, revealing a third time scale. Treating $\varepsilon^2$ as the new fast time scale, the secondary layer problem features a set of equilibria $S_1$, with its biologically relevant critical manifold $S_{1,1}$ given by the parabolic curve:
\begin{align} \label{eq:fold_tumour}
   {S}_{1,1} = \left\{ (M,C) \in \mathbb{R}^2_{\geq 0} \, \Bigg| \, C = \frac{\beta_2}{\gamma_2} M(1-M) \right\}.
\end{align}
The nontrivial eigenvalue of the secondary layer problem along ${S}_{1,1}$ is $\lambda^{(1)} = \beta_2 (1-2 M)(M+ \delta)$. Thus, ${S}_{1,1}$ is attracting for $M > \frac{1}{2}$ and repelling for $M < \frac{1}{2}$. Normal hyperbolicity is lost at $M_F = \frac{1}{2}$, which corresponds to a generic saddle-node bifurcation in the layer problem:
\begin{lemma} \label{lemma:fold_point}
    The point $(M,C)_F = \left(\frac{1}{2}, \frac{\beta_2}{4\gamma_2}\right) $ is a fold point of the layer problem associated with \eqref{reduced_again2}. 
\end{lemma}
\begin{proof}
    We verify the standard non-degeneracy conditions for a fold with $C$ as the bifurcation parameter: evaluated at $(M,C)_F$, the vector field and its first $M$-derivative vanish, while the transversality condition $D_C f_{0}^{(1)} = \gamma_2(\delta + \frac{1}{2}) \neq 0$ and non-degeneracy condition $D_{MM} f_{0}^{(1)} = 2\beta_2 (\delta + \frac{1}{2}) \neq 0$ hold.
\end{proof}
\paragraph*{Slow dynamics on the nested 1D manifold ${S}_{1,1}$}
Using the parametrisation method once more, we compute the $\mathcal{O}(\varepsilon^3)$ flow on the 1D manifold ${S}_{1,1}$, which is given by
\begin{align}
    \dfrac{dM}{d t} &=  \varepsilon^3 \dfrac{H(M)}{1-2M}  
    =  \varepsilon^3 \Bigg( \dfrac{(M-1)(\beta_1 M - \alpha_2 M + \beta_1 \delta + \gamma_1 M^2 + \delta \gamma_1 M)}{1-2M} + \dfrac{\alpha_1 \gamma_2 (M+\delta)}{\beta_2 (1-2M)} \Bigg). \label{reduced_scalingB_1D1}
\end{align}
The roots of \eqref{reduced_scalingB_1D1} yield up to three equilibria $M_s, M_{l,-}, M_{l,+}$ corresponding to the full system's small-tumour state $\mathcal{EQ}_s$ and large-tumour states $\mathcal{EQ}_{l,-}, \mathcal{EQ}_{l,+}$ respectively. Figure \ref{scalingB_1D_bifurcation_E0} displays the 1D bifurcation diagram for \eqref{reduced_scalingB_1D1}. 

{By combining the stability of these roots with the transversal eigenvalue signatures of the critical manifold, we map the full 3D stability: (i) $\mathcal{EQ}_s = (-,+,+)$ and (ii) $\mathcal{EQ}_{l,+} = (-,-,-)$. Here, the first entry corresponds to the attraction toward $S_{0,E}^\varepsilon$, the second entry depends on whether the equilibrium lies on the attracting or repelling branch of $S_{1,1}$, and the third entry corresponds to the stability on the reduced 1D problem \eqref{reduced_scalingB_1D1}. The symbol `$+$' indicates a positive eigenvalue and `$-$' indicates a negative real part eigenvalue.}
{Note, the eigenvalue signature of the equilibrium $\mathcal{EQ}_{l,-}$ depends on its position relative to the fold $M_F = 1/2$. If $M_{l,-} < M_F$, it lies on the repelling branch of $S_{1,1}$ and is stable relative to the reduced system \eqref{reduced_scalingB_1D1}, yielding the signature $\mathcal{EQ}_{l,-} = (-,+,-)$. Conversely, if $M_{l,-} > M_F$, it lies on the attracting branch of $S_{1,1}$ and is unstable in the reduced system, yielding $\mathcal{EQ}_{l,-} = (-,-,+)$.}

{Applying the time desingularisation $d\tilde{\tau} = (1-2M)^{-1} d\tau$ removes the singularity at $M_F$ (while reversing the flow on the attracting branch $M > \frac{1}{2}$), yielding the desingularised vector field:}
\begin{align}
    \dfrac{dM}{d\tilde{\tau}} &=  \varepsilon^3 H(M) \label{reduced_scalingB_1D}.
\end{align}
\begin{figure}[ht]
\centering
\includegraphics[width=0.55\linewidth]{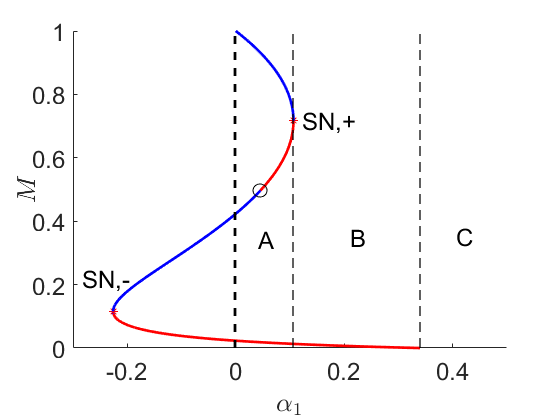}
\caption{Bifurcation diagram of the 1D reduced system \eqref{reduced_scalingB_1D1} via MatCont \cite{matcont}, featuring two saddle-node bifurcations (SN). Blue/red lines denote stable/unstable equilibria on the 1D manifold. The circled point denotes a folded singularity, occurring when $M_{l,-}$ crosses the fold point $M_F$. {(Negative $\alpha_1$ values are shown in order to identify the second saddle-node bifurcation.) Parameter values are as in Table \ref{parameter_values_new} where $\alpha_1$ is the bifurcation parameter.}}
\label{scalingB_1D_bifurcation_E0}
\end{figure}

\section{Singular bifurcations}
\label{sec:bifs11}
The bifurcations of equilibria $M_{l,+}, M_{l,-}, M_s$ within the 1D reduced problem \eqref{reduced_scalingB_1D1} govern the onset of excitable and oscillatory dynamics in the full 3D system \eqref{full_single_epsilon}. In this section, we uncover these singular bifurcations, which partition the dynamics of parameter space. We assume throughout that trajectories crossing the conjugate $C$-axis and the origin are governed by a well-defined global return mechanism; the rigorous proof of this passage via blow-up analysis is deferred to Section \ref{sec:blowup_tumour}.

\subsection{Saddle-node, cusp, and SNIC bifurcations}
\label{sec:boundaryAB}
In Region A ($0 \leq \alpha_1 < \alpha_{1,SN,+}$), the system exhibits excitable dynamics. The equilibria have values $M_s< M_{l,-} < M_{l,+}$ on ${S}_{1,1}$, where the case $M_{l,-}<M_F$ is shown in Figure \ref{fig:excitable}a. A singular trajectory $\Gamma_0$ (Figure \ref{fig:excitable}a) initialised away from the equilibria undergoes a transient large excursion, spending $\mathcal{O}(\varepsilon^2)$ time near both ${S}_{0,E}$ and ${S}_{0,M}$, before asymptotically settling on the stable large-tumour state $M_{l,+}$. Biologically, this represents a temporary reduction in tumour load followed by inevitable relapse, creating an excitable medium where perturbations trigger large excursions.
\begin{figure}[ht]
\centering
\begin{subfigure}{0.45\textwidth}
\centering
  \includegraphics[width=1\linewidth]{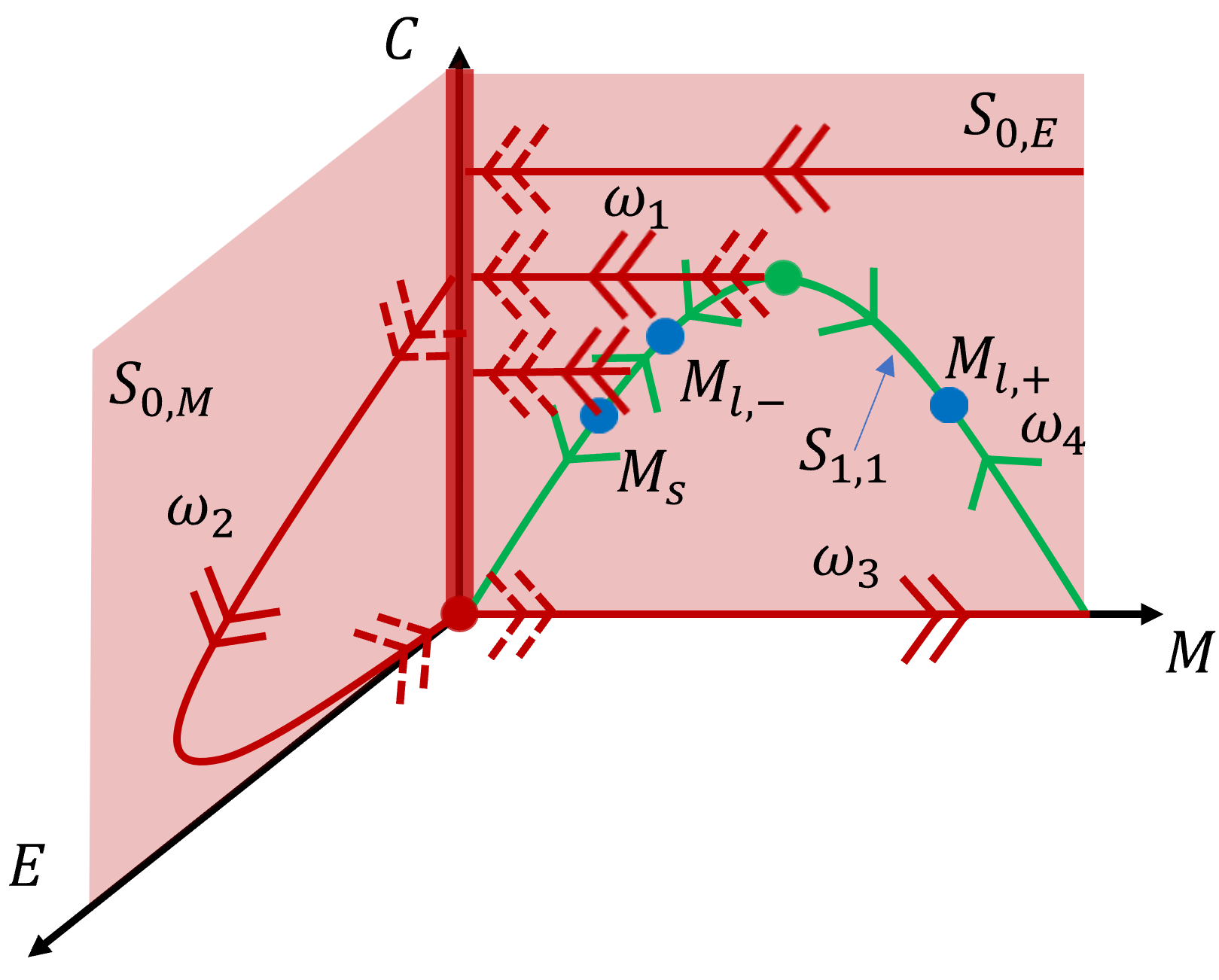}
  \caption{}
\end{subfigure}
\begin{subfigure}{0.45\textwidth}
\centering
  \includegraphics[width=1\linewidth]{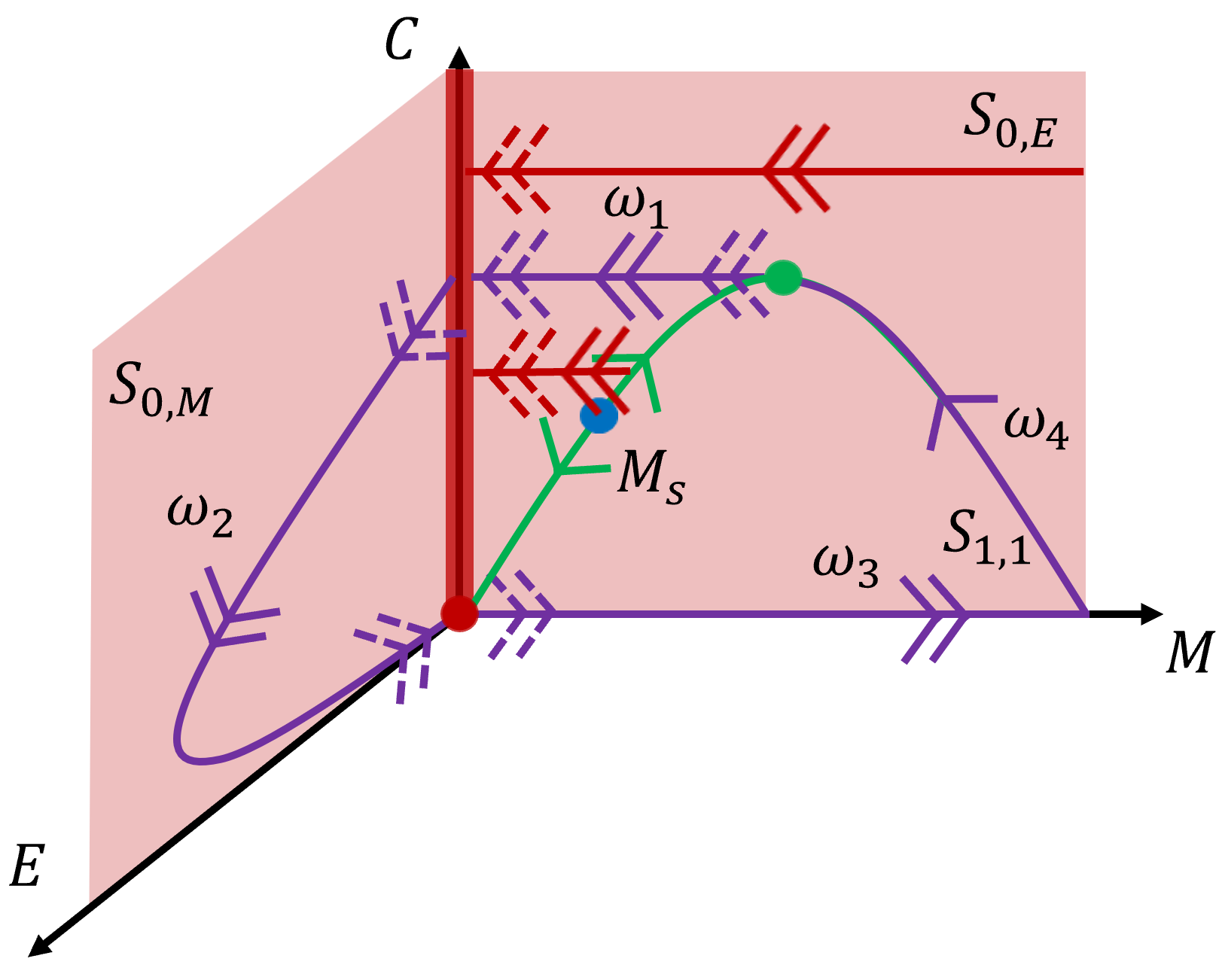}
  \caption{}
\end{subfigure}
\caption{Phase space geometry of system \eqref{full_single_epsilon}. Red planes denote the critical manifolds, with the bolded $C$-axis and origin marking the loss of normal hyperbolicity. Singular solutions are formed by concatenating segments of the reduced and layer problems. (a) A transient singular excursion (excitable dynamics) settling onto the large-tumour equilibrium $M_{l,+}$. {(b) A singular relaxation oscillation (purple) forming a closed loop.}}
\label{fig:excitable}
\end{figure}

As $\alpha_1$ increases, $M_{l,+}$ and $M_{l,-}$ may collide and annihilate via a saddle-node (SN) bifurcation. Algebraically, this occurs when the desingularised vector field \eqref{reduced_scalingB_1D} satisfies the equilibrium condition $H = 0$ and the non-hyperbolic condition $D_M H = 0$. Solving $D_M H = 0$ for $\alpha_1$ gives the constraint:
\begin{align}
\label{eq:alpha1_eq}
\alpha_1(M) = -\frac{\beta_{2} [\alpha_{2} - \beta_{1} - 2M\alpha_{2} + 2M\beta_{1} - 2M\gamma_{1} + \beta_{1}\delta - \delta\gamma_{1} + 3M^2\gamma_{1} + 2M\delta\gamma_{1}]}{\gamma_{2}}.
\end{align}
Extending this analysis to the two-parameter plane $(\alpha_1, \beta_1)$, where $\beta_1$ is proportional to the effector cell death rate, we solve $H=0$ and $D_MH=0$ simultaneously. This parametrises the SN locus by $M$ as:
\begin{align} \label{eq:SN_curve}
\alpha_1(M) &= \frac{\beta_{2}(M-1)^2 (\gamma_{1}M^2+2\gamma_{1}M\delta +\gamma_{1}\delta^2-\alpha_{2}\delta)}{\gamma_{2}(M+\delta)^2}, \nonumber \\
\beta_1(M) &= \alpha_{2}+\gamma_{1}-2M\gamma_{1}-\frac{\alpha_{2}\delta(\delta +1)}{(M+\delta)^2}.
\end{align}

\begin{lemma} \label{prop:sn_cusp}
Given the parameters in Table \ref{parameter_values_new}, the system exhibits two saddle-node bifurcations at $(\alpha_1,M)_{SN,-} \approx (-0.226, 0.116)$ and $(\alpha_1,M)_{SN,+} \approx (0.106, 0.717)$. Furthermore, in the $(\alpha_1,\beta_1)$ parameter plane, the SN curves meet at a cusp bifurcation at $M_{cusp} \approx 0.272$, corresponding to $(\alpha_1,\beta_1) \approx (0.518, 0.745)$.
\end{lemma}

\begin{proof}
The locations follow directly from finding the roots of the derivative of \eqref{eq:alpha1_eq} (for the SN), and the simultaneous roots of $\alpha_1'(M) = \beta_1'(M) = 0$ (for the cusp). The conditions $D_{MM} H \neq 0$ and  $D_{\alpha_1} H \neq 0$ (for the SN), and $D_{MMM} H^{(3)} > 0$ with transversal determinant $\det J > 0$ (for the cusp) are calculated at these coordinates, confirming they are non-degenerate.
\end{proof}

Note, if the saddle-node collision at $M_{SN,+}$ occurs on the \emph{attracting} branch of ${S}_{1,1}$ (i.e., $M_{SN,+} > M_F = \frac{1}{2}$), the global return mechanism (Figure \ref{fig:excitable}b) maps the unstable manifold of the saddle $M_{l,-}$ directly back to the stable node $M_{l,+}$. Thus, the SN bifurcation manifests as a \emph{saddle-node on invariant circle (SNIC)} bifurcation.

\begin{lemma} \label{lemma:SNIC}
For $\beta_1 \in (\beta_1^{lower}, \beta_1^{upper})$, the right SN branch $(\alpha_1,M_{SN,+})_{SN,+}$ is a SNIC bifurcation, where:
\begin{align*}
    \beta_1^{upper} &= \frac{\alpha_{2}}{4(\delta + 1/2)^2}, \quad 
    \beta_1^{lower} = \alpha_{2}+\gamma_{1}+\delta\gamma_{1}-2\sqrt{\alpha_{2}\gamma_{1}(\delta +1)}.
\end{align*}
\end{lemma}
\begin{proof} 
The upper bound corresponds to the SN occurring exactly at the fold point $M_F = 1/2$. Substituting $M=1/2$ into \eqref{eq:SN_curve} yields $\beta_1^{upper}$. The lower bound occurs when the small-tumour equilibrium $M_s \to 0$ as $\alpha_1 \to \frac{\beta_1 \beta_2}{\gamma_2}$ from the left, which disrupts the closed oscillatory loop (where the flow hits the non-smooth boundary at the origin). Equating the $\alpha_1 \to \frac{\beta_1 \beta_2}{\gamma_2}$ limit with the SN curve \eqref{eq:SN_curve} yields $\beta_1^{lower}$. Within this interval, the SN resides on the attracting manifold with a global return mechanism, satisfying the geometric conditions for a SNIC.
\end{proof}

\subsection{Singular Andronov-Hopf and singular Bogdanov-Takens bifurcations}
\label{sec:singularAH}
In the planar fast-slow subsystem \eqref{reduced_again2}, the collision of an equilibrium on $S_{1,1}$ with the fold point gives rise to a singular Andronov-Hopf (sAH) bifurcation \cite{krupaszmolyan2001fold}, which permits trajectories to temporarily track the repelling branch of ${S}_{1,1}$ via singular canards.

\begin{corollary} \label{cor:FS1}
Setting $M = 1/2$ in the equilibrium condition $H = 0$ yields the parameter locus for the sAH bifurcation of the layer problem \eqref{reduced_again2} at:
\begin{align} \label{eq:folded_sing_eq_param}
    \alpha_{1,sAH} = \frac{\beta_{2}(2\beta_{1}-2\alpha_{2}+\gamma_{1}+4\beta_{1}\delta +2\delta\gamma_{1})}{4\gamma_{2}(2\delta +1)}
\end{align} given that $\beta_1 > \beta_1^{upper} = \frac{\alpha_{2}}{4(\delta + 1/2)^2}$.
\end{corollary}

By transforming \eqref{reduced_again2} into the canonical normal form of Krupa and Szmolyan \cite{krupaszmolyan2001fold,canardexplosion}, we deduce the existence of maximal canards and a canard explosion for $0 < \varepsilon \ll 1$. The first Lyapunov coefficient $A$ of this sAH is strictly positive, implying a subcritical sAH bifurcation, though higher-order $\varepsilon$-corrections may alter its criticality \cite{singularAH_correction}. {Further details can be found in Section SM1 of Supplementary Material II.}

When the saddle-node collision $M_{SN,+}$ coincides exactly with the fold point, we encounter a codimension-two \emph{singular Bogdanov-Takens (sBT)} bifurcation. 

\begin{lemma} \label{prop:sBT}
The point $(\alpha_1, \beta_1)_{sBT} = \left( \frac{\beta_{2}(\gamma_{1}/4-\alpha_{2}\delta +\delta\gamma_{1}+\delta^2\gamma_{1})}{4\gamma_{2}(\delta + 1/2)^2}, \frac{\alpha_{2}}{4(\delta + 1/2)^2}\right)$ constitutes a singular Bogdanov-Takens point on the right SN curve \eqref{eq:SN_curve}.
\end{lemma}
\begin{proof}
This point corresponds to $M_{SN,+} = M_F = 1/2$, which is exactly the upper bound $\beta_1^{upper}$, which we now relabel as $\beta_{1,sBT}$ established in Lemma \ref{lemma:SNIC}. Routine calculation confirms the non-degeneracy condition $\partial^2_{MM} \mathcal{G} > 0$ (where $\mathcal{G}$ is the $C$-equation of system \eqref{reduced_again2}), and the transversality of the Jacobian determinant with respect to $(\alpha_1, \beta_1)$ is non-zero, satisfying the generic sBT conditions detailed in \cite{neural}. 
\end{proof}

\subsection{Singular bifurcation diagram and canard explosions}
\label{sec:singular_bif1}
The interplay of the SN, cusp, and sBT bifurcations partitions the parameter space into distinct dynamical regions (Figure \ref{fig:cusp_regime}). Notably, for $\beta_1 > \beta_{1,sBT}$, the right SN curve is no longer a SNIC because the collision occurs on the repelling branch ($M_{SN,+} < M_F$). Instead, the transition between excitable (Region A') and oscillatory (Region B') dynamics is governed by the sAH bifurcation.
\begin{figure}[ht]
\centering
\includegraphics[width=0.55\linewidth]{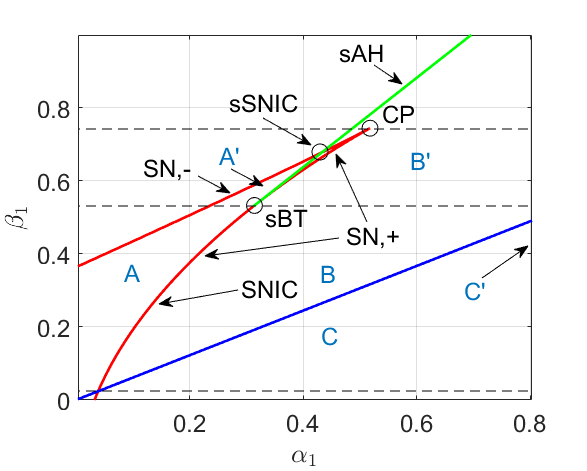}
\caption{Partial singular bifurcation diagram of equilibria on ${S}_{1,1}$ in $(\alpha_1, \beta_1)$ space. The green curve corresponds to the sAH/homoclinic bifurcation emerging from the sBT point. The red curve is the SN locus, a subset of which is a SNIC. The blue curve given by $\alpha_1 = \frac{\beta_1\beta_2}{\gamma_2}$ denotes a global bifurcation determined via the blow-up analysis in Section \ref{sec:blowup_tumour}.}
\label{fig:cusp_regime}
\end{figure}
For $0 < \varepsilon \ll 1$, the sAH bifurcation unfolds into a \emph{canard explosion}. Within an exponentially small parameter window of $\alpha_1$, the small-amplitude AH limit cycles rapidly grow into singular canards before expanding into full relaxation oscillations (Figure \ref{fig:canard_full}). Due to the subcritical nature of the sAH, this explosion involves a saddle-node of periodic orbits (SNPO). 

\begin{figure}[ht]
\centering
\includegraphics[width=0.475\linewidth]{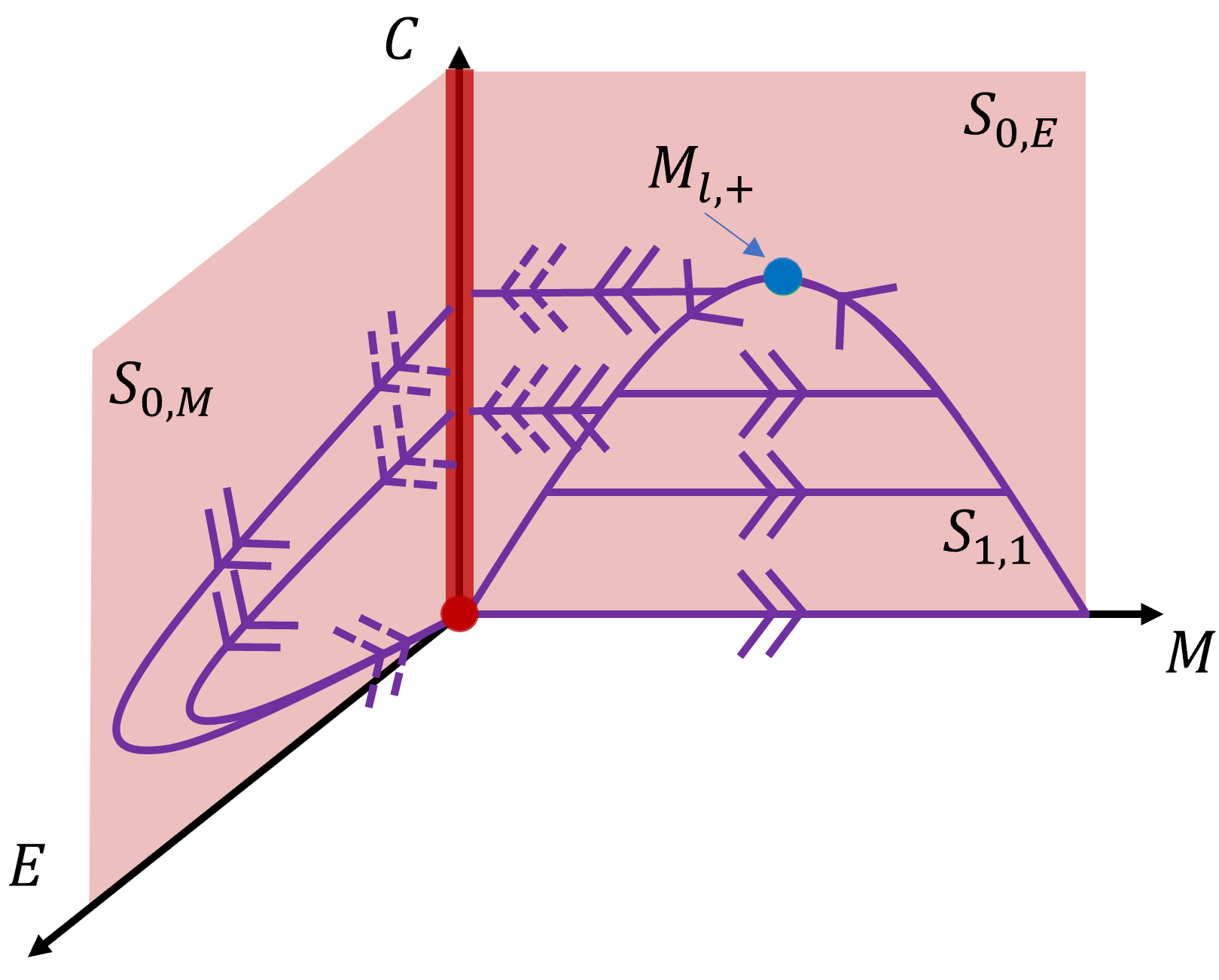}
\caption{Family of singular canard cycles growing to full relaxation oscillations, forming a closed orbit around $\mathcal{EQ}_{l,+} = (E_{l,+},M_{l,+},C_{l,+})$.}
\label{fig:canard_full}
\end{figure}

Depending on the proximity to the saddle $M_{l,-}$, this canard explosion may be \emph{incomplete}. The expanding canard cycles can collide with the stable manifold of $M_{l,-}$, forming a homoclinic orbit (Figure \ref{fig:canard_incomplete}) that instantly destroys the periodic orbit before a full relaxation oscillation can form. Thus, within the subset of Region A', which we call Region A'.I, canard cycles exist, whereas in Region A'.II, oscillatory dynamics are prohibited due to $M_{l,-}$ blocking the necessary closed loop. 

\begin{figure}[ht]
\centering
\begin{subfigure}{0.475\textwidth}
\centering
\includegraphics[width=1\linewidth]{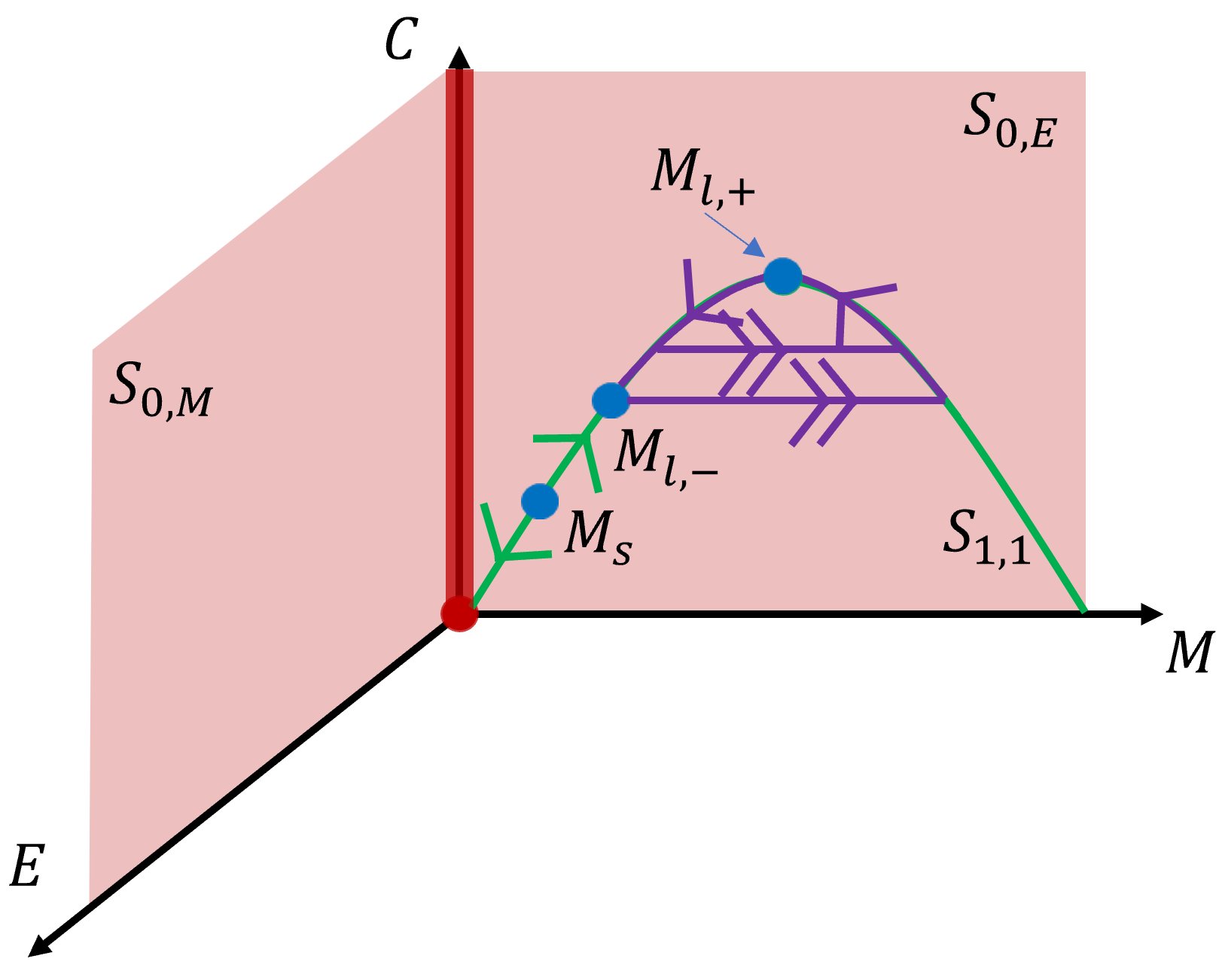}
  \caption{}
\end{subfigure}
\begin{subfigure}{0.475\textwidth}
\centering
\includegraphics[width=1\linewidth]{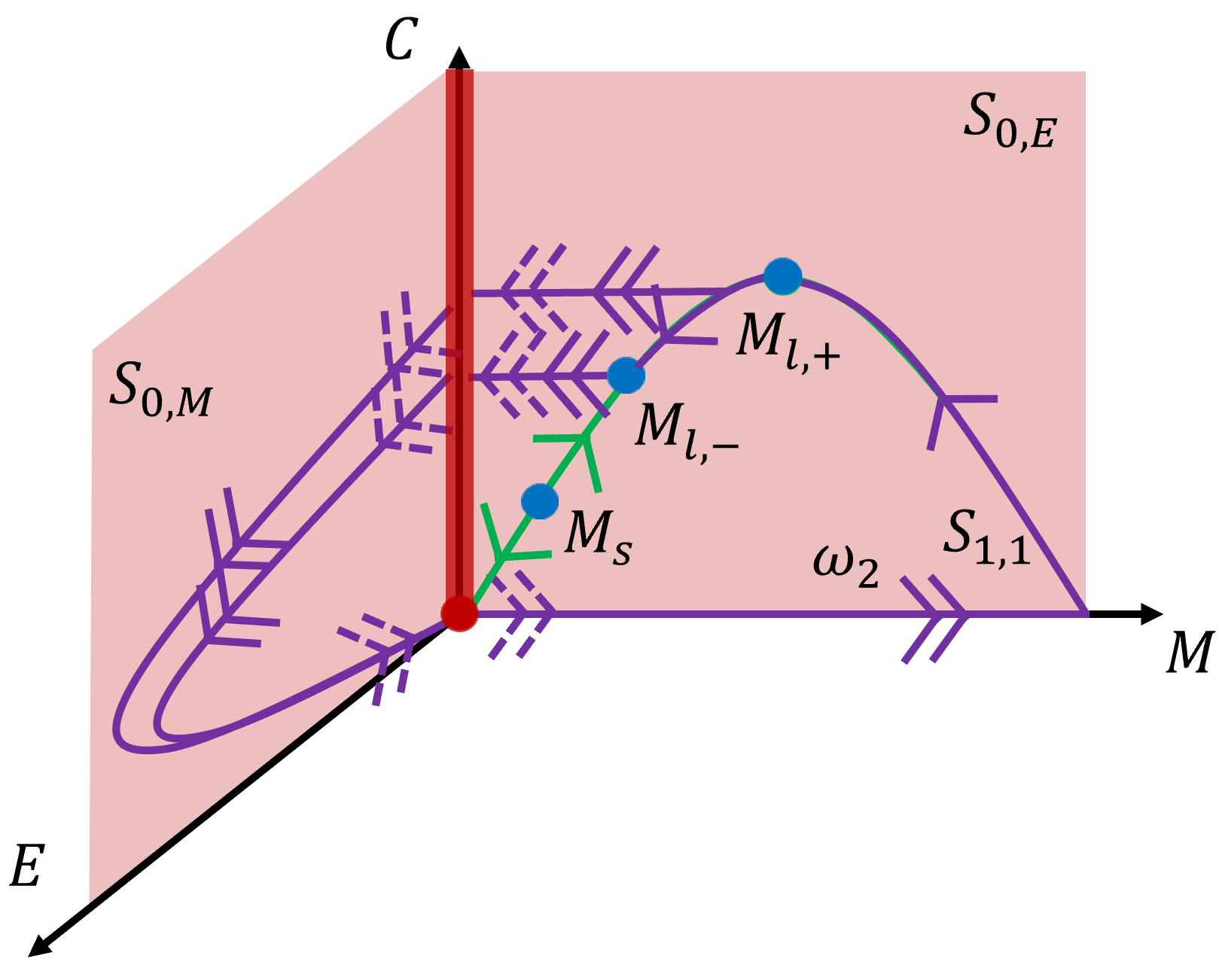}
  \caption{}
\end{subfigure}
\caption{Examples of singular canard cycles disrupted by a homoclinic loop. (a) Canards without head cycles terminating at a small homoclinic orbit of $M_{l,-}$. (b) Canards with head cycles terminating at a large homoclinic orbit of $M_{l,-}$.}
\label{fig:canard_incomplete}
\end{figure}
To complete this global geometric picture and formalise the return mechanisms assumed throughout this section, we must rigorously resolve the flow at the degenerate intersection on the conjugate $C$-axis and at the origin.

\section{Blow-up analysis} \label{sec:blowup_tumour}
To rigorously establish the global return mechanism, specifically the passage of trajectories through the non-hyperbolic regions at the $C$-axis ($C>0$) and the origin, we employ the geometric blow-up technique for singularly perturbed systems \cite{dumortierroussarie1996}. By desingularising these intersections, we prove the existence of stable relaxation oscillations in Regions B and B', and excitable large excursions in Regions A and A'. We focus primarily on proving the existence of unique relaxation oscillations in Region B, followed by a brief discussion on extending these arguments to the other parameter regimes.

\subsection{Augmented system and main theorem}
We begin by appending the trivial dynamics of the perturbation parameter $\varepsilon$ to system \eqref{full_single_epsilon}:
\begin{align}
\begin{pmatrix}
E' \\
M' \\
C' \\
\varepsilon'
\end{pmatrix} &= \begin{pmatrix}
    -1 \\ -1 \\ 1 \\ 0
\end{pmatrix} EM(M + \delta ) + \varepsilon \begin{pmatrix} 1 \\ 1 \\ -1 \\ 0
\end{pmatrix} (M+\delta) C + \mathcal{O}(\varepsilon^2).
\label{full_single_epsilon2}
\end{align}
where prime denotes the derivative with respect to time $t$. 
Since the flow traverses two distinct degenerate regions, we employ two separate blow-up transformations on \eqref{full_single_epsilon2}:
\begin{enumerate}
    \item A \emph{cylindrical blow-up} $\Phi : \mathbb{R}^2 \times S^{2} \to \mathbb{R}^{4}$ to resolve the passage from ${S}_{0,E}$ to ${S}_{0,M}$ across the $C$-axis ($C > 0$).
    \item A \emph{spherical blow-up} $\bar{\Phi} : \mathbb{R} \times S^{3} \to \mathbb{R}^{4}$ to resolve the passage from ${S}_{0,M}$ back to ${S}_{0,E}$ across the origin.
\end{enumerate}
Applying these transformations to the singular relaxation oscillation $\Gamma_0$ (from Figure \ref{fig:excitable}b) gives a desingularised loop $\Gamma_0 = \omega_1 \cup \omega_{12} \cup \omega_2 \cup \omega_{23} \cup \omega_3 \cup \omega_4$, where the non-hyperbolic jumps are replaced by hyperbolic segments $\omega_{12}$ and $\omega_{23}$ across the blown-up manifolds. This geometry is illustrated in Figure \ref{blow_up_pic_sing}.

\begin{figure}[ht]
\centering
 \includegraphics[width=0.8\linewidth]{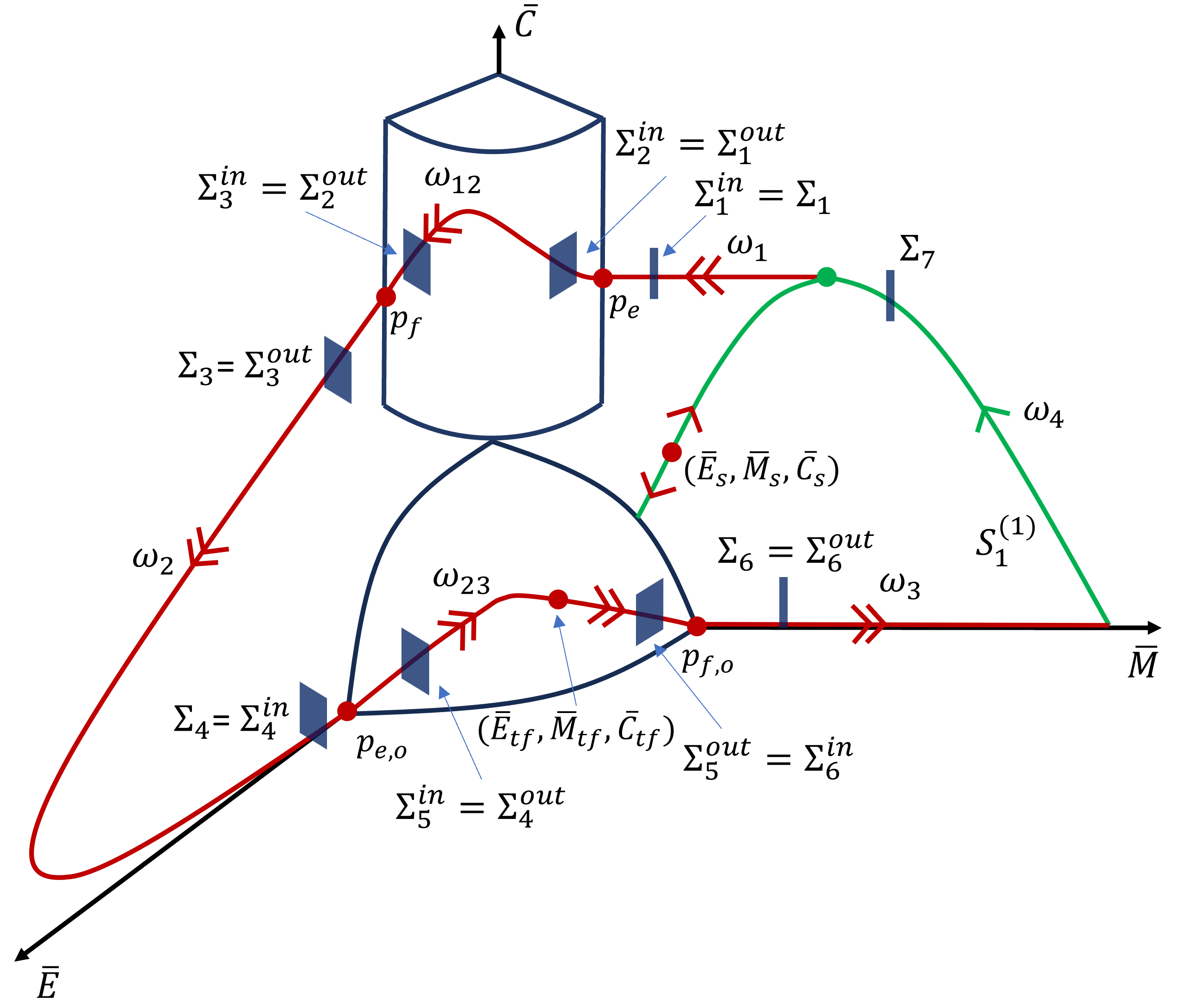}
  \caption{Schematic diagram of the singular relaxation oscillation mapped into the blown-up space in the singular limit $\bar{\varepsilon} = 0$. The non-hyperbolic segments $\omega_{12}$ and $\omega_{23}$ are resolved inside the cylinder (for the $C$-axis) and the sphere (for the origin), respectively. The entry and exit points for the spherical blow-up are denoted with $o$.}
\label{blow_up_pic_sing}
\end{figure}

\begin{theorem} \label{main_theorem}
Let the parameters be fixed to the values in Table \ref{parameter_values_new} (except $\alpha_1, \beta_1$). Fix $\beta_1 \in ( \beta_1^{lower}, \beta_{1,sBT})$ and choose $\alpha_1 \in (\alpha_{1,SN,+}(\varepsilon,\beta_1), \phi(\varepsilon,\beta_1))$, where $\phi = \frac{\beta_1\beta_2}{\gamma_2} - \mathcal{O}(\varepsilon)$, placing the system inside Region B. For sufficiently small $\varepsilon > 0$, there exists a unique, exponentially stable relaxation oscillation $\Gamma_\varepsilon$ for system \eqref{full_single_epsilon}.
\end{theorem}

\begin{remark} \label{C_problem}
Rather than utilising a successive blow-up procedure (cf.~\cite{kosiukszmolyan,kosiukszmolyan2,kosiukszmolyan3,process}), we blow up the $C$-axis and the origin independently. This approach is valid because the singular jump from ${S}_{0,M}$ to ${S}_{0,E}$ occurs at $C = \frac{\beta_2}{4\gamma_2} > 0$, bounded away from the origin.
\end{remark}

\subsection{Resolution of the $C$-axis (cylindrical blow-up)}
\label{blow_up_C}
To track the flow through the cylindrical blow-up $\Phi(r, \bar{\varepsilon}, \bar{E},\bar{M},\bar{C}) = (r \bar{E}, r \bar{M}, \bar{C},r^2 \bar{\varepsilon})$, we define three local coordinate charts: an entry chart $K_1$, a rescaling chart $K_2$, and an exit chart $K_3$. The entry chart $K_1$ captures the exponential contraction of the incoming fast flow onto a centre manifold, funneling it into the rescaling chart $K_2$. Symmetrically, $K_3$ guides the flow departing the cylinder toward the stable slow manifold ${S}_{0,M}^\varepsilon$. The detailed vector fields for $K_1$ and $K_3$ are deferred to Section SM2 of Supplementary Material II.

\subsubsection{Rescaling chart $K_2$ dynamics}
\label{sec:K2_C}
The core dynamics connecting the two manifolds occur in the rescaling chart $K_2$. Applying the local transformation $E = r_2 E_2, \, M = r_2 M_2, \, C = C_2, \, \varepsilon = r_2^2$ and a time desingularisation $t_{2} = r_2 t$ yields:
\begin{align}
\begin{pmatrix}
\frac{dE_2}{dt_{2}} \\
\frac{dM_2}{dt_{2}} \\
\frac{dC_2}{dt_{2}}
\end{pmatrix} &= \begin{pmatrix}
    \delta \\ \delta \\ 0
\end{pmatrix} (C_2- E_2M_2) + r_2 \begin{pmatrix}
    M_2 \\ M_2 \\ -\delta
\end{pmatrix}(C_2 - E_2 M_2) + r_2^2 \begin{pmatrix}
   0 \\  -\delta \gamma_2 C_2 \\ -M_2C_2 + E_2M_2^2
\end{pmatrix}  + \mathcal{O}(r_2^3).
\label{full_single_epsilon_tilde}
\end{align}
Taking the singular limit $r_2 \to 0$ reveals an attracting critical manifold in the biologically revelant domain: 
$$\hat{{S}}_0 = \{ (E_2,M_2,C_2) \in \mathbb{R}^3_{\geq 0} \mid C_2 = E_2 M_2 \}\,.$$ 
This manifold $\hat{{S}}_0$ acts as a smooth bridge, connecting the $\mathcal{O}(\varepsilon)$-extensions of the original slow manifolds (${S}_{0,E}^\varepsilon$ and ${S}_{0,M}^\varepsilon$) across the previously non-hyperbolic $C$-axis (see Figure \ref{3D_layer_scalingD}a). 
\begin{figure}[ht]
\centering
\begin{subfigure}{0.425\textwidth}
\centering
 \includegraphics[width=1\linewidth]{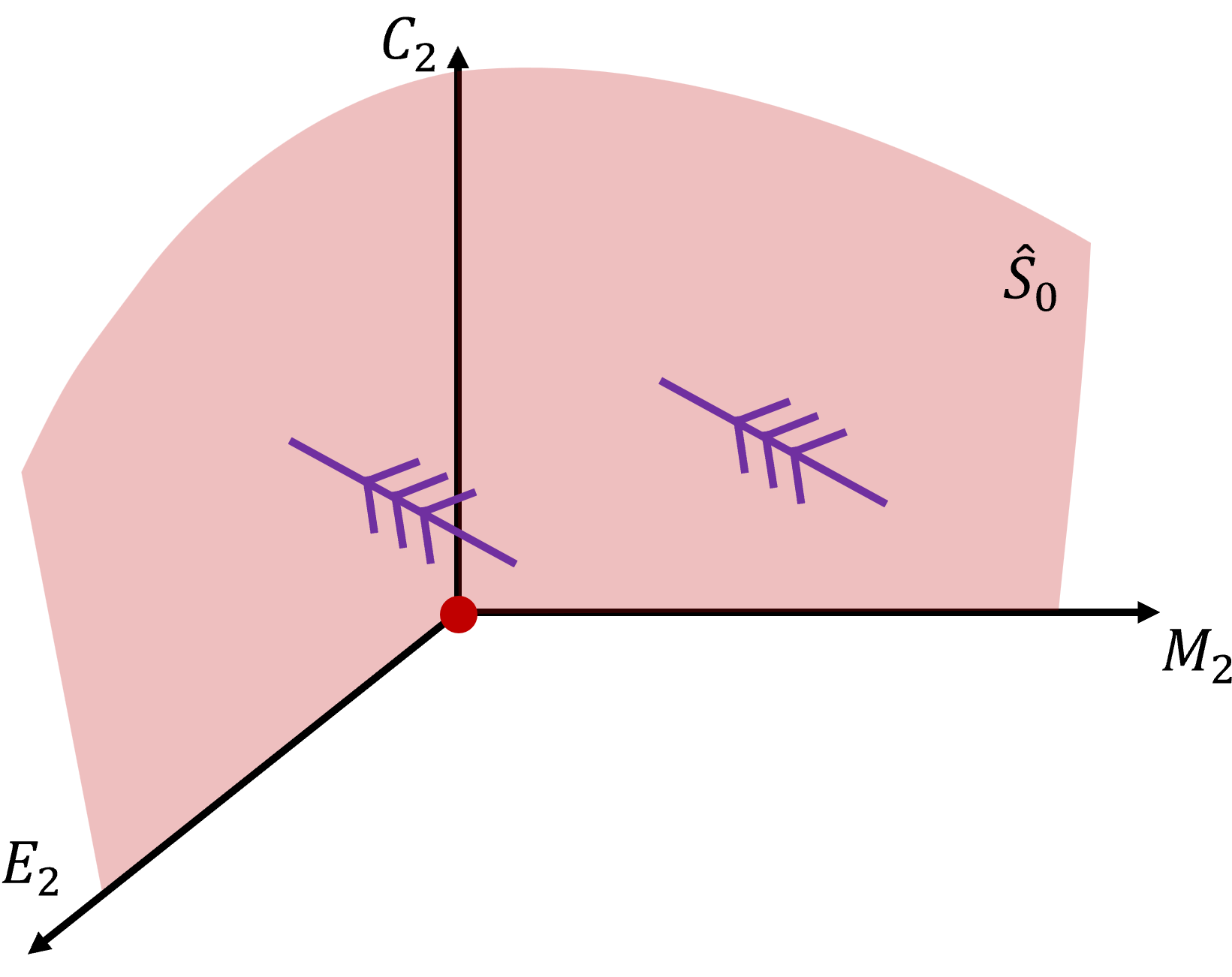}
  \caption{}
\end{subfigure}
\hspace{0.5cm}
\begin{subfigure}{0.4\textwidth}
\centering
 \includegraphics[width=1\linewidth]{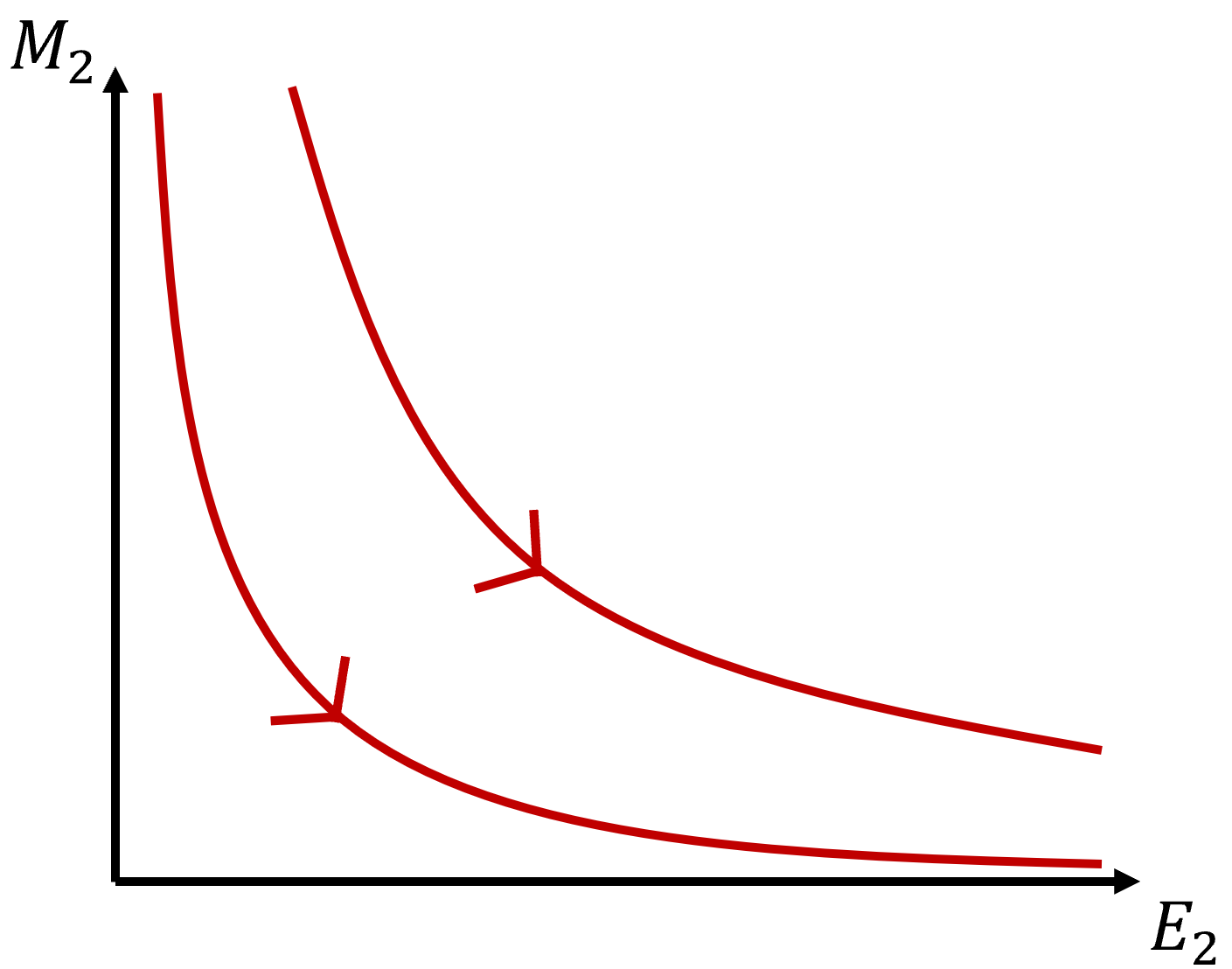}
  \caption{}
\end{subfigure}
  \caption{(a) The layer problem in the rescaling chart $K_2$. The critical manifold $\hat{{S}}_0$ smoothly connects the incoming and outgoing flow. The red dot indicates a loss of normal hyperbolicity (b) Sketch of the flow induced by the slow vector field \eqref{slow_flow_D}. The solution curves are invariant hyperbolas for $C_2 > 0$ fixed.}
\label{3D_layer_scalingD}
\end{figure}
Applying the parametrisation method, the $\mathcal{O}(r_2^2)$-slow vector field on $\hat{{S}}_0^{r_2}$ dictates this passage:
\begin{align}
    \begin{pmatrix}
        \frac{dE_2}{dt_{2}} \\ \frac{dM_2}{dt_{2}} 
    \end{pmatrix} = r_2^2 \frac{\delta \gamma_2 E_2 M_2}{E_2+M_2} \begin{pmatrix}
        E_2 \\ -M_2
    \end{pmatrix} + \mathcal{O}(r_2^3). \label{slow_flow_D}
\end{align}
The solution curves satisfy $M_2(E_2) = D/E_2$ (for constant $D>0$), forming invariant hyperbolas (Figure \ref{3D_layer_scalingD}b). Thus, trajectories enter from $M_2 \to \infty, E_2 \to 0$ and transit smoothly to $M_2 \to 0, E_2 \to \infty$, successfully crossing the $C$-axis.

\subsection{Resolution of the degenerate origin (spherical blow-up)}
\label{blow_up_orig}
To resolve the origin, we apply the spherical blow-up $\bar{\Phi}(r, \bar{\varepsilon}, \bar{E},\bar{M},\bar{C}) = (r \bar{E}, r \bar{M}, r\bar{C},r \bar{\varepsilon})$, again utilising three coordinate charts: the entry chart $K_4$, the rescaling chart $K_5$, and the exit chart $K_6$. The detailed vector fields for $K_4$ and $K_6$ are deferred to Section SM2 of Supplementary Material II.
\subsubsection{Rescaling chart $K_5$ dynamics}
\label{sec:K2_origin}
In the rescaling chart $K_5$ defined by $E = r_5 E_5, M = r_5 M_5, C = r_5 C_5, \varepsilon = r_5$, we apply the time scale $t_5 = \varepsilon t$. The layer limit ($r_5 \to 0$) reveals an attracting critical manifold $\tilde{{S}}_0 = \{(E_5,M_5,C_5) \in \mathbb{R}^3_{\geq 0} \mid C_5=E_5M_5\}$. The entire fast dynamics near the origin has now gained normal hyperbolicity.

Projecting the dynamics onto $\tilde{{S}}_0$ yields the reduced slow flow:
\begin{align}
    \begin{pmatrix}
    \frac{dE_5}{dt_5} \\
    \frac{dM_5}{dt_5}
    \end{pmatrix} &= \dfrac{\varepsilon \delta}{E_5+M_5+1} \begin{pmatrix}
        \gamma_2 E_5^2 M_5 - \beta_1  E_5^2 - \beta_2  E_5M_5 +  (\alpha_1 -\beta_1) E_5 + \alpha_1 \\ -\gamma_2 E_5M_5^2 +  (\beta_1 - \gamma_2) E_5M_5 + \beta_2  M_5^2 + (\beta_2 - \alpha_1) M_5
    \end{pmatrix}. \label{scalingC_reduced}
\end{align}
This system admits an equilibrium at $\tilde{\mathcal{EQ}}_{tf} = (\frac{\alpha_1}{\beta_1}, 0)$, which is a saddle for $\alpha_1 <\frac{\beta_1\beta_2}{\gamma_2}$, and stable for $\alpha_1 >\frac{\beta_1\beta_2}{\gamma_2}$. This corresponds to the tumour-free equilibrium $\mathcal{EQ}_{tf}$ of system \eqref{full_single_epsilon}. The stable manifold of $\tilde{\mathcal{EQ}}_{tf}$ aligns with the invariant subspace $W^s_{5,sub} = \{M_5 = 0\} \cap \tilde{{S}}_0$, which smoothly extends the tumour-free incoming flow from the entry chart $K_4$, and the tumour-free invariant subspace $W^s_{sub}$ in the slow system \eqref{reduction_M0}.

The geometric mechanism of the origin passage is governed by the saddle $\tilde{\mathcal{EQ}}_{tf} = (\frac{\alpha_1}{\beta_1}, 0)$ for $\alpha_1 <\frac{\beta_1\beta_2}{\gamma_2}$. In the singular limit, trajectories enter $K_5$ along $W^s_{5,sub}$, approach $\tilde{\mathcal{EQ}}_{tf}$, and are immediately repelled along its unstable manifold $W^u_5$. As sketched in Figure \ref{K2_origin_eps0}b, the geometry of the $E_5$ and $M_5$ nullclines funnels $W^u_5$ toward the $M_5$-axis (i.e., $E_5 \to 0, M_5 \to \infty$), matching the required entry conditions for the exit chart $K_6$, which returns the flow to ${S}_{0,E}^\varepsilon$. The equations for the nullclines are found in Section SM2 of Supplementary Material II.

\begin{remark}
    The corresponding extension of the slow manifold $S_{0,M}^\varepsilon$ in the exit chart $K_3$ and the entry chart $K_4$ are the centre manifolds $W^c_3$ and $W^c_4$, respectively. It turns out that correction terms beyond the leading-order terms near the equator (i.e., $\{\varepsilon_3 = r_3 = 0\}$ and $\{\varepsilon_4 = r_4 = 0\}$) are required for the power series representations of the centre manifolds $W^c_3$ and $W^c_4$. This is to ensure that the vector fields on the centre manifolds match with the slow vector field \eqref{reduction_M0} on $S_{0,M}^\varepsilon$. This slow vector field \eqref{reduction_M0} was calculated using the parametrisation method, which showcases its necessity for a successful blow-up application. We refer the reader to Remarks SM2.5 and SM2.8 in Supplementary Material II for further details.
\end{remark}

\begin{figure}[ht]
\centering
\begin{subfigure}{0.425\textwidth}
\centering
 \includegraphics[width=1\linewidth]{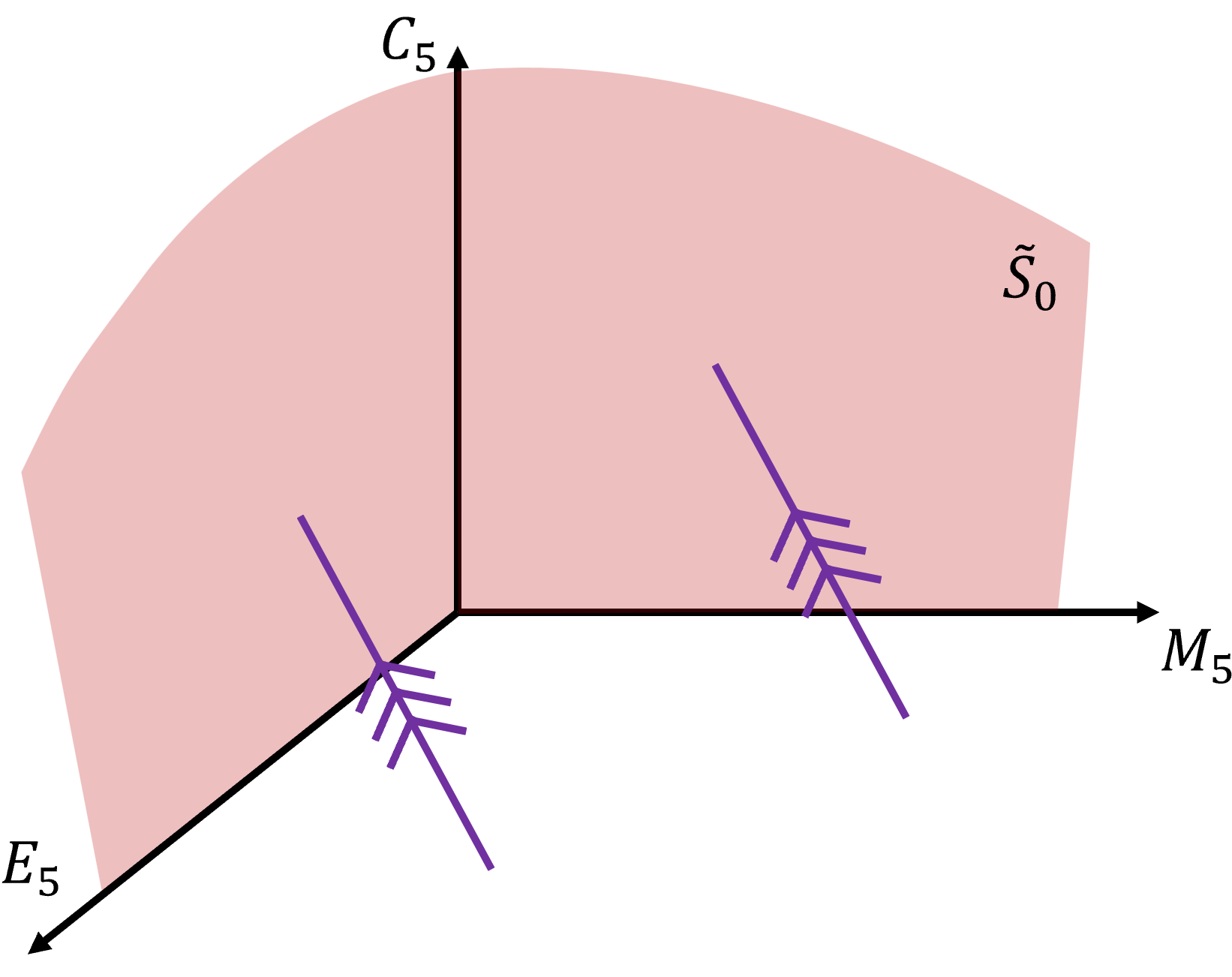}
  \caption{}
\end{subfigure}
\hspace{0.5cm}
\begin{subfigure}{0.4\textwidth}
\centering
 \includegraphics[width=1\linewidth]{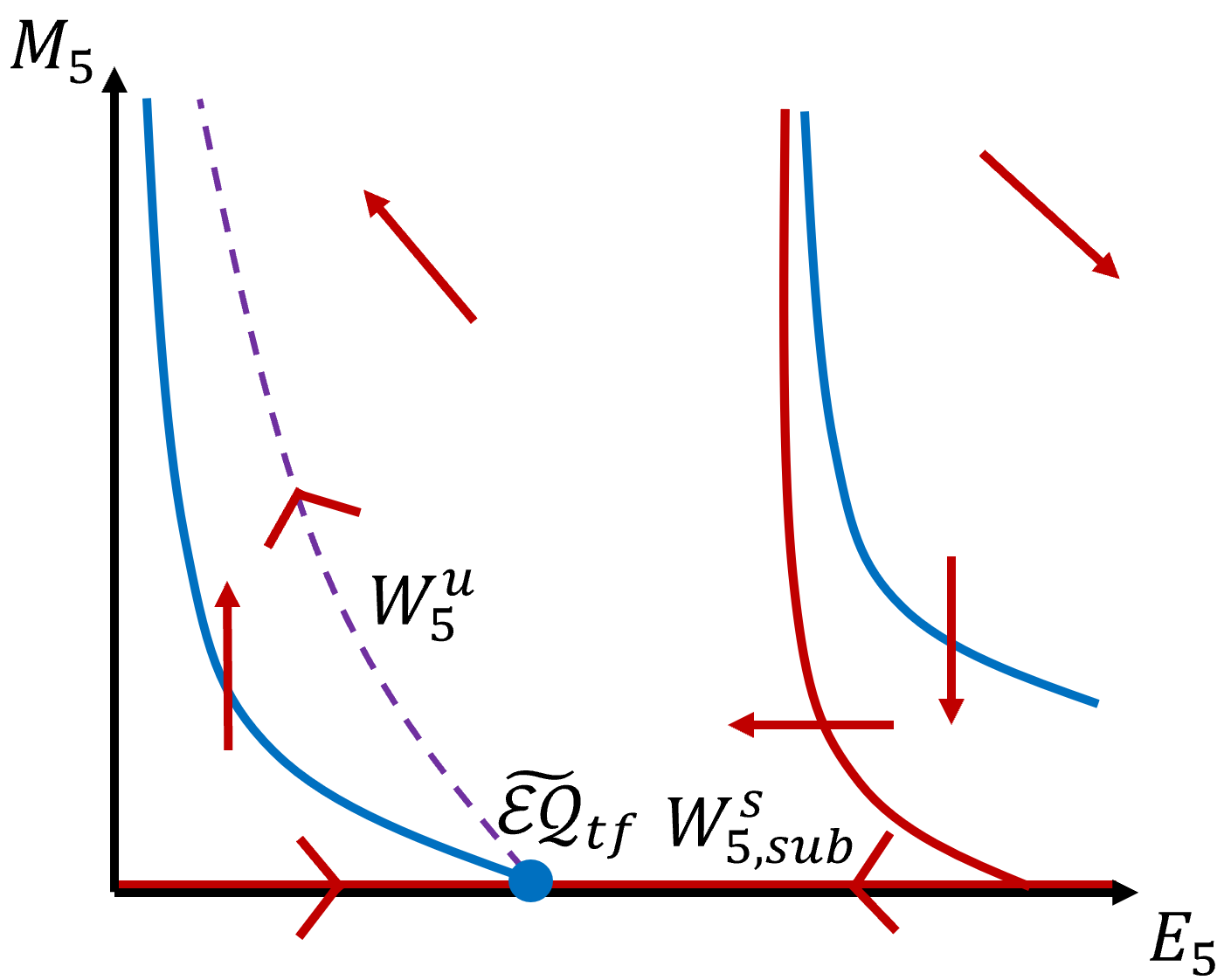}
  \caption{}
\end{subfigure}
  \caption{(a) Sketch of the layer problem the rescaled system in chart $K_5$. (b) Phase portrait of the reduced flow \eqref{scalingC_reduced} on $\tilde{{S}}_0$ for $0 < \alpha_1  < \frac{\beta_1 \beta_2}{\gamma_2}$. The stable manifold (red line) brings the flow into the saddle $\tilde{\mathcal{EQ}}_{tf}$. The nullcline geometry (blue/red curves) funnels the unstable manifold (purple dotted line) toward the $M_5$-axis, guaranteeing exit.}
\label{K2_origin_eps0}
\end{figure}

\subsection{Proof of Theorem \ref{main_theorem}}
With the local geometry fully resolved, the global return mechanism is formalised by composing the sequential local transition maps along the singular cycle $\Gamma_0$:
\begin{align*}
    \Pi = \Pi_7 \circ \Pi_6 \circ \Pi_4 \circ \Pi_3 \circ \Pi_1\,,
\end{align*}
where $\Pi_1$ represents the composite map through the entire cylindrical blow-up (mapping across charts $K_1 \to K_2 \to K_3$), $\Pi_4$ represents the composite map through the spherical blow-up (mapping across charts $K_4 \to K_5 \to K_6$), and $\Pi_3, \Pi_6, \Pi_7$ track the slow and fast flows between these degenerate regions. We refer the reader to Section SM3 in Supplementary Material II for the rigorous decomposition of these maps using local intra-chart flows and transition diffeomorphisms.

\begin{proof}[Proof of Theorem \ref{main_theorem}]
By Tikhonov-Fenichel theory and the hyperbolicity established inside the blow-up charts, each local map $\Pi_i$ is a well-defined diffeomorphism for sufficiently small $\varepsilon > 0$. As detailed in Propositions SM3.1, SM3.5, and SM3.8 in Supplementary Material II, the strong contraction onto the normally hyperbolic slow manifolds ${S}_{0,M}^\varepsilon$ and ${S}_{1,1}^\varepsilon$ dominates any mild expansion near the saddles. Therefore, the composition $\Pi$ is a strict contraction mapping on the cross-section $\Sigma_1$. By the \textit{Contraction Mapping Theorem}, $\Pi$ admits a unique, exponentially stable fixed point, proving the existence of the attracting limit cycle $\Gamma_\varepsilon$.

To prove uniqueness \textit{globally}, we note that the basin of attraction for the stable branch of ${S}_{1,1}^\varepsilon$ acts as a global funnel. Any constructed relaxation oscillation must then eventually pass through the cross-section $\Sigma_1$. The contraction of $\Pi$ guarantees that any initial condition in this basin will converge to the unique limit cycle $\Gamma_\varepsilon$, rendering all such oscillations exponentially close to each other.
\end{proof}

\subsection{Extensions to other parameter regions (regions A, A' and B')}
\label{sec:extend}
The geometric framework developed above extends naturally to prove the observed dynamics in adjacent parameter regimes. 

For Region B', if $\beta_1 \in (\beta_{1,sBT}, \gamma_2)$, the contraction arguments remain identical, establishing stable relaxation oscillations until $\alpha_1$ decreases to intersect the homoclinic or SNPO boundaries discussed in Section \ref{sec:singular_bif1}. If $\beta_1 \in (\gamma_2, \beta_{1,cusp})$, the entry geometry in chart $K_4$ shifts: the incoming flow approaches a different saddle equilibrium $p_{other,4}$ rather than the original entry point. However, trajectories passing $p_{other,4}$ are still guided towards $W^s_{5,sub}$ in chart $K_5$, preserving the global contraction.

Finally, the same blow-up maps confirm the existence of the transient large excursions defined in Section \ref{sec:boundaryAB} for Regions A and A'. For these parameters, the flow successfully traverses the cylinder and the sphere but eventually maps into the basin of attraction of the stable node $M_{l,+}$ on the stable branch of $S_{1,1}^\varepsilon$, rather than forming a closed periodic loop. This rigorous topological connection justifies the excitable behaviour observed numerically.

\section{Singular bifurcations at the origin: the organising centre}
\label{sec:degenerate bifurcations}
In this section, we uncover the topological organising centre that governs the boundary of oscillatory dynamics detailed in Theorem \ref{main_theorem} given by $\alpha_1 \in \left(\alpha_{1,SN,+},\frac{\beta_1\beta_2}{\gamma_2} - \mathcal{O}(\varepsilon)\right)$, and its relation to the unidentified bifurcation in Figure \ref{fig:cusp_regime}. Our analysis identifies another sAH bifurcation, and a singular transcritical (sTC) bifurcation. {We ultimately conjecture these are organised by a novel \emph{singular transcritical Bogdanov-Takens (stBT)} bifurcation.} 

\begin{remark}
    For brevity in this section, we drop the subscript $5$ from the spherical blow-up chart $K_5$.
\end{remark}

\subsection{The nilpotent Jacobian}
Recall that $\tilde{\mathcal{EQ}}_{tf}$ of system \eqref{scalingC_reduced} changes stability at the critical parameter value $\alpha_1 = \frac{\beta_1\beta_2}{\gamma_2}$. At this parameter value, system \eqref{scalingC_reduced} may be rewritten as
\begin{align}
    \begin{pmatrix}
    \frac{dE}{dt} \\
    \frac{dM}{dt}
    \end{pmatrix} &= \varepsilon \bar{N}_{0} \bar{f}_{0} = \dfrac{\varepsilon \delta}{\gamma_2(E+M+1)}\begin{pmatrix}
        \beta_1 + \beta_1 E - \gamma_2 E M \\ (\gamma_2 - \beta_1 + \gamma_2 M)M
    \end{pmatrix} \left( \beta_2 - \gamma_2 E\right). \label{scalingC_reduced2}
\end{align}
The layer problem possesses a critical manifold $\tilde{S}_{0,2} = \left\{ (E,M) \in \mathbb{R}^2 \, | \, E= \frac{\beta_2}{\gamma_2}\right\}$, and its associated nontrivial eigenvalue undergoes a loss of normal hyperbolicity at $M = \frac{\beta_1(\beta_2 + \gamma_2)}{\beta_2 \gamma_2}$.

\begin{lemma} \label{lemma:sBT(B)}
    $(E,M)_{sBT(B)} = \left(\frac{\beta_2}{\gamma_2},\frac{\beta_1(\beta_2 + \gamma_2)}{\beta_2 \gamma_2}\right)$ is a contact point of order one for the layer problem of system \eqref{scalingC_reduced2}.
\end{lemma}
\begin{proof}
    Evaluating the Jacobian $D\bar{f}_0$ at $(E,M)_{sBT(B)}$ yields a rank-1 matrix. Calculations involving the Hessian with the left and right nullvectors (as detailed in Lemma 4.1 of \cite{wechselberger2020}) yields a strictly non-zero quantity $\frac{\beta_1\beta_2 \delta^2 \gamma_2}{(\beta_1+\beta_2)(\beta_2 + \gamma_2)} \neq 0$, satisfying the generic conditions for an order-one contact point.
\end{proof}
In order to check if $(E,M)|_{{sBT(B)}}$ is a pseudo-singularity, which would imply an sAH bifurcation in the $\varepsilon > 0$ unfolding, we need to obtain the reduced system on $\tilde{S}_{0,2}$. Thus, we must compute an $\mathcal{O}(\varepsilon)$-correction to system \eqref{scalingC_reduced} using the parametrisation method:
\begin{align}
    \begin{pmatrix}
    \frac{dE}{dt} \\
    \frac{dM}{dt}
    \end{pmatrix} &= \tilde{R}_1 + \varepsilon \tilde{R}_2  \nonumber \\
    &=\dfrac{\delta}{E+M+1} \begin{pmatrix}
        \gamma_2 E^2 M - \beta_1  E^2 - \beta_2  EM +  (\alpha_1 -\beta_1) E + \alpha_1 \\ -\gamma_2 EM^2 +  (\beta_1 - \gamma_2) EM + \beta_2  M^2 + (\beta_2 - \alpha_1) M
    \end{pmatrix} + \varepsilon \begin{pmatrix}
        \tilde{R}_{2,1} \\ \tilde{R}_{2,2} 
    \end{pmatrix}, \label{scalingC_reduced_again}
\end{align}
where we have rescaled time $t \to t/\varepsilon$ and the correction terms $\tilde{R}_{2,1}$ and $\tilde{R}_{2,2}$ are rational functions of $(E,M)$; see Section SM2 of Supplementary Material I.
{
The reduced system at $\alpha_1 =  \frac{\beta_1\beta_2}{\gamma_2}$ on the critical manifold $\tilde{S}_{0,2}$ is then given by
\begin{align}
    \dfrac{dM}{dt} = -\frac{\beta _{2}\gamma_2 M^2 \,\left(\alpha _{2}+\beta _{1}\delta -\delta \gamma _{1}\right) }{\beta _{1}\beta _{2}+\beta _{1}\gamma _{2}-\beta _{2}\gamma _{2} M}.
\end{align}}
{It turns out that this is only a regular jump point.} 
{\begin{lemma}
    Given $\alpha_1 =\frac{\beta_1\beta_2}{\gamma_2}$ and $\beta_1 > 0$, $(E,M)_{sBT(B)} = \left(\frac{\beta_2}{\gamma_2}, \frac{\beta_1\beta_2 + \beta_1 \gamma_2}{\beta_2\gamma_2} \right)$ is a regular jump point of the layer problem of system \eqref{scalingC_reduced_again}.
\end{lemma}}

{\begin{proof}
    From Lemma \ref{lemma:sBT(B)}, $(E,M)|_{sBT(B)}$ is a contact point of order one. The final condition (see Definition 4.4 in \cite{wechselberger2020}) is
    \begin{align*}
        \bar{N}_0 \adj{(D\bar{f}_0 \bar{N}_0)} D \bar{f}_0 \tilde{R}_2 = \begin{pmatrix}
            0 \\ -\frac{\beta_1^2 \delta ( \beta_2 + \gamma_2) ( \alpha_2 + \beta_1 \delta - \delta \gamma_1)}{\gamma_2(\beta_1+\beta_2)}
        \end{pmatrix} \neq \begin{pmatrix} 0\\0\end{pmatrix}\,,
    \end{align*} i.e., we have a regular jump point.
\end{proof}}

However, evaluating the Jacobian of this fully expanded $\mathcal{O}(\varepsilon)$-corrected system at the coordinates $(E,M) = (E,M)_{sBT(B)}$ and parameters $(\alpha_1,\varepsilon) =\left(\frac{\beta_1\beta_2}{\gamma_2},0\right)$ leads directly to a nilpotent structure,
\begin{align} \label{J_BT}
    J &= \begin{pmatrix}
        0 & 0 \\  -\frac{\beta_1 \delta \gamma_2}{\beta_2} & 0
    \end{pmatrix}\,,
\end{align}
which is the geometric hallmark of a Bogdanov-Takens (BT) bifurcation. 

\begin{lemma} \label{lemma:typeB}
    For $\varepsilon > 0$ sufficiently small, applying a near-identity coordinate shift translates system \eqref{scalingC_reduced_again} near $(E,M)_{sBT(B)}$ into a modified BT normal form:
    \begin{align} \label{eq:AH_not_normal}
    \begin{pmatrix}
        \frac{d\xi_1}{dt} \\
        \frac{d \xi_2}{dt} 
    \end{pmatrix} = \begin{pmatrix}
          \xi_2 \\ \tilde{\mu}_1 + \tilde{\mu}_2 \xi_1 + \tilde{A} \xi_1^2 + \tilde{B} \xi_1 \xi_2
    \end{pmatrix},
    \end{align}
    where $(\tilde{\mu}_1,\tilde{\mu}_2) \to (0,0)$ as $\left(\alpha_1 - \frac{\beta_1\beta_2}{\gamma_2}, \varepsilon\right) \to (0,0)$.
\end{lemma}
\begin{proof}
    Applying a Taylor expansion up to second order near contact point $(E,M) = (E,M)_{sBT(B)}$ satisfies the generic $a_{20} + b_{11} \neq 0$ BT condition, which are certain quadratic coefficients; see \cite{kuznetsov_bifurcation}. The parameter map  $(\alpha_1,\varepsilon) \to (\tilde{\mu}_1,\tilde{\mu}_2)$ is regular for $\beta_1 > 0$, guaranteeing the existence of the modified normal form \eqref{eq:AH_not_normal}.
\end{proof}

Since $\tilde{A}$ vanishes at the singular limit $\left(\alpha_1 - \frac{\beta_1\beta_2}{\gamma_2}, \varepsilon\right) \to (0,0) \to (0,0)$, classical smooth normal form transformations break down, reflecting the inherently singular nature of this bifurcation. Nevertheless, analysing the unfolded system reveals an AH bifurcation emanating from this nilpotent point, and defer these calculations to Section SM4 of Supplementary Material II. We note that that its first Lyapunov coefficient is $\tilde{l}_1 > 0$ for $\beta_1 \gtrsim 0$ and $\varepsilon >0$, resulting in a \textit{subcritical} AH.

{In the sBT framework of De Maesschalck and Dumortier \cite{sfBT}, our system corresponds to Case 1b of their Figure 8, where an Andronov-Hopf (AH) bifurcation emanates from the sBT point for $\varepsilon > 0$. This justifies the suffix ``(B)'' in sBT(B) and AH(B) to denote the specific restricted unfolding from this nilpotent BT-like point. We note that capturing this unfolding is fundamentally a codimension-2 problem, as both parameters $(\alpha_1,\varepsilon)$ must be varied to trace the emergence of the AH(B) curve in parameter space.} 

\subsection{Resolution by parameter rescaling}
The previous analysis indicates that at $\alpha_1 = \frac{\beta_1\beta_2}{\gamma_2}$ there is a nilpotent BT-like point, as well as a change in stability of $\tilde{\mathcal{EQ}}_{tf}$ (see Section \ref{sec:K2_origin}) occurring at distinct $M$-coordinates. 
Let us apply the translation $(E,M,\alpha_1) \to \left( E+\frac{\beta_2}{\gamma_2}, M, \alpha_1 + \frac{\beta_1 \beta_2}{\gamma_2}\right)$ to system \eqref{scalingC_reduced_again} and rescale time $dt\to H(E,M)dt$ to remove denominators. Applying $\alpha_1 \to \varepsilon \alpha_1$ then gives

\begin{align} \label{eq:translated_TC}
\begin{pmatrix} 
    \frac{dE}{dt} \\
    \frac{dM}{dt}
    \end{pmatrix} &= \dfrac{D^3 E}{\beta_2 }\begin{pmatrix} \beta_2(\beta _{2}{\gamma _{2}}^2 E M +\beta _{1} {\gamma _{2}}^2 E+ {\beta _{2}}^2\gamma _{2} M)
\\ -\left(\beta _{1}\beta _{2}+\beta _{1}\gamma _{2}+\beta _{2}\gamma _{2} M\right)\left(\beta_1\gamma _{2} + \beta_2 \gamma_2 +\beta _{2}\gamma _{2} M \right) \end{pmatrix} + \varepsilon \tilde{\tilde{F}}_1(E,M,\alpha_1) +\mathcal{O}(\varepsilon^2)
\end{align}
where $D = \beta _{1}\beta _{2}+\beta _{1}\gamma _{2}+\beta _{2}\gamma _{2}+{\beta _{2}}^2+\beta _{2}\gamma _{2} E+\beta _{2}\gamma _{2} M$. 
{
The system is singularly perturbed, with the critical manifold $\tilde{\tilde{S}}_0 := \{E= 0\}$ and corresponding nontrivial eigenvalue
\begin{align*}
    D\tilde{\tilde{f}}_0 \tilde{\tilde{N}}_0|_{\tilde{\tilde{S}}_0} = -\dfrac{\left(\beta _{1}\beta _{2}+\beta _{1}\gamma _{2}-\beta _{2}\gamma _{2} M\right){\left(\beta _{2}+\gamma _{2}+\gamma _{2} M\right)}^3}{{\gamma _{2}}^4}
\end{align*}
which is attracting for $0 \leq M<M_{sBT(B)} = \frac{\beta_1(\beta_2 + \gamma_2)}{\beta_2 \gamma_2}$. There is a loss of normal hyperbolicity at $M= M_{sBT(B)}$, which is a contact point of order one:
\begin{lemma} \label{lemma:sAH}
    Given $\beta_1 > 0$, $(E,M)|_{sBT(B)} = \left(0,\frac{\beta_1(\beta_2 + \gamma_2)}{\beta_2 \gamma_2}\right)$ is a contact point of order one of the layer problem of system \eqref{eq:translated_TC}.
\end{lemma}
\begin{proof}
    The proof is similar to that of Lemma \ref{lemma:sBT(B)}. Indeed, the rank of $D \tilde{\tilde{f}}_0$ is 1, and evaluating the second non-degeneracy condition (the Hessian bilinear form applied to the left and right nullvectors, as detailed in Lemma 4.1 of \cite{wechselberger2020}) gives
    \[
        -\frac{\left(\beta _{1}\beta _{2}+\beta _{1}\gamma _{2}\right)\left(\gamma _{2}-\beta _{1}+\frac{\beta _{1}\beta _{2}+\beta _{1}\gamma _{2}}{\beta _{2}}\right){\left(\beta _{2}+\gamma _{2}+\frac{\beta _{1}\beta _{2}+\beta _{1}\gamma _{2}}{\beta _{2}}\right)}^6}{{\gamma _{2}}^7} \neq 0.
    \]
\end{proof}

The reduced problem on the critical manifold $\tilde{\tilde{S}}_0$ is then given by:
\begin{align} \label{eq:reduced_TC}
 \dfrac{dM}{dt} = \varepsilon  \tilde{\tilde{R}}_1 = -\varepsilon \frac{M{\left(\beta _{2}+\gamma _{2}+M\gamma _{2}\right)}^4\left(M\alpha _{2}\beta _{2}+\alpha _{1}\delta \gamma _{2}+M\beta _{1}\beta _{2}\delta -M\beta _{2}\delta \gamma _{1}\right)}{\delta {\gamma _{2}}^3\left(\beta _{1}\beta _{2}+\beta _{1}\gamma _{2}-M\beta _{2}\gamma _{2}\right)}.
\end{align}
This system admits two relevant equilibrium branches:
$$\tilde{\tilde{\mathcal{EQ}}}_1 = \{M=0\},\quad \tilde{\tilde{\mathcal{EQ}}}_2 = \left\{M = -\frac{\alpha _{1}\delta \gamma _{2}}{\alpha _{2}\beta _{2}+\beta _{1}\beta _{2}\delta -\beta _{2}\delta \gamma _{1}} \right\}\,.$$
\begin{lemma} \label{lemma:TC}
    For a fixed $\beta_1 > 0$, system \eqref{eq:reduced_TC} undergoes a transcritical (TC) bifurcation at the point $(M,\alpha_1) = (0,0)$.
\end{lemma}
\begin{proof}
    Evaluating the vector field \eqref{eq:reduced_TC} at $(M,\alpha_1) = (0,0)$ confirms $D_{M} \tilde{\tilde{R}}_1 = D_{\alpha_1} \tilde{\tilde{R}}_1 = 0$, while the Hessian determinant with respect to $(M, \alpha_1)$ is strictly negative and $D_{MM} \tilde{\tilde{R}}_1 \neq 0$, satisfying the standard non-degeneracy conditions for a TC bifurcation.
\end{proof}
Furthermore, the contact point of order one in Lemma \ref{lemma:sAH} is now a pseudo-singularity:
\begin{corollary} \label{label:cor_TC} $(E,M,\alpha_1)|_{{sAH(B)}} = \left(0, \frac{\beta_1\beta_2 + \beta_1 \gamma_2}{\beta_2\gamma_2},-\frac{\beta _{1}\left(\beta _{2}+\gamma _{2}\right)\left(\alpha _{2}+\beta _{1}\delta -\delta \gamma _{1}\right)}{\delta {\gamma _{2}}^2} \right)$  is a \\ pseudo-singularity of the layer problem of system \eqref{eq:translated_TC}.
\end{corollary}
\begin{proof}
    From Lemma \ref{lemma:sAH}, $(E,M)= \left(0, \frac{\beta_1\beta_2 + \beta_1 \gamma_2}{\beta_2\gamma_2}\right)$ is a contact point of order one. Furthermore, for the given $\alpha_{1,sAH(B)}$, we have that $\tilde{\tilde{N}}_0 \adj{(D\tilde{\tilde{f}}_0 \tilde{\tilde{N}}_0)} D \tilde{\tilde{f}}_0 \tilde{\tilde{F}}_1 = \mathbf{0}$, breaking the condition for a regular jump point; see \cite{wechselberger2020}.
\end{proof}}

{Thus, the parameter rescaling successfully separates the previously identified sBT(B) point from the bifurcation governing the stability change of $\tilde{\mathcal{EQ}}_{tf}$ in system \eqref{scalingC_reduced}. In particular, the equilibrium branch $\tilde{\tilde{\mathcal{EQ}}}_2$ now forms a slanted line in the $(\alpha_1, M)$-plane (Figure \ref{fig:bif_blow_up}), intersecting the origin via a singular TC (sTC) bifurcation. While this sTC bifurcation remains at its original parameter value ($\alpha_{1,sTC} = \alpha_{1,TC}$), the parameter shift transforms the contact point into a true pseudo-singularity located at a distinct $\alpha_1$-value. Following the theory of Krupa and Szmolyan \cite{krupaszmolyan2001fold}, the emergence of this pseudo-singularity confirms that the AH(B) bifurcation emanates from the sBT(B) point.}

\begin{figure}[ht]
\centering
\begin{subfigure}{0.375\textwidth}
\centering
 \includegraphics[width=0.9\linewidth]{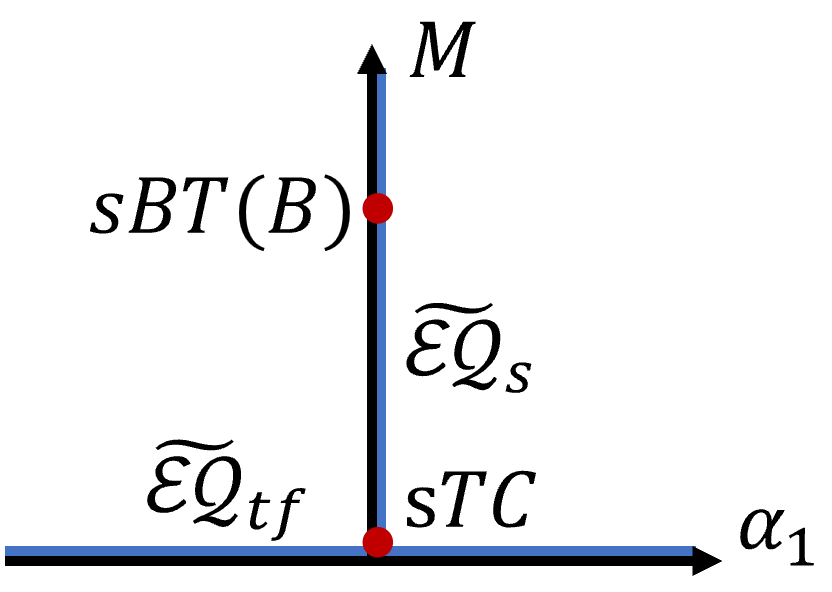}
  \caption{}
\end{subfigure}
\begin{subfigure}{0.375\textwidth}
\centering
  \includegraphics[width=0.9\linewidth]{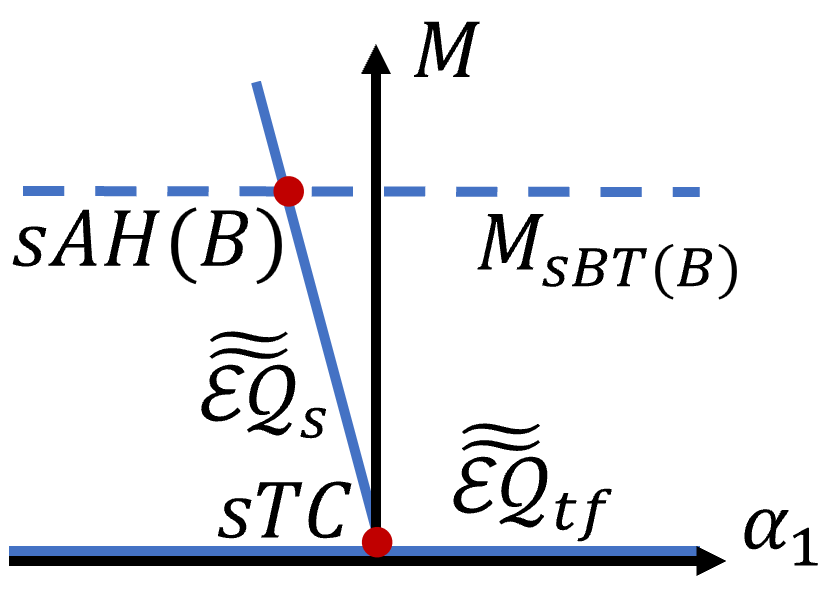}
  \caption{}
\end{subfigure}
\caption{(a) The singular limit before parameter rescaling: the sTC and sBT(B) collide parametrically at $\alpha_1=0$. (b) Bifurcation diagram after the parameter rescaling $\alpha_1 \to \varepsilon \alpha_1$, separating the sTC from the pseudo-singularity/sAH(B). Here $M_{sBT(B)} = \frac{\beta_1\beta_2+\beta_1\gamma_2}{\beta_2\gamma_2}$.}
\label{fig:bif_blow_up}
\end{figure}

Since the parameter blow-up $\alpha_1 \to \varepsilon \alpha_1$ was required to parametrically separate the sTC and the sBT(B), we conclude that the AH(B) bifurcation is $\mathcal{O}(\varepsilon)$-close to the TC bifurcation, satisfying the boundary limit $\phi = \alpha_{1,sTC} - \mathcal{O}(\varepsilon)$ invoked in Theorem \ref{main_theorem}.

\subsubsection{Conjecture: singular transcritical Bogdanov-Takens}
\label{sec:singular_BT_tumour}
{Since the sBT(B)/sAH(B) point spatially and parametrically collides with the sTC point in system \eqref{scalingC_reduced_again} as $(\alpha_1, \beta_1,\varepsilon) \to (0, 0,0)$ (see Figure \ref{fig:bif_blow_up}, where the sBT(B)/sAH(B) complex converges to the sTC), we conjecture that the overarching degenerate structure is a \textit{singular transcritical BT bifurcation (stBT)}. This degenerate intersection acts as a central organising centre when unfolded by the effector supply rate $\alpha_1$, the effector death rate $\beta_1$, and the singular parameter $\varepsilon$.}
\begin{lemma} \label{lemma:BT_normal_sf}
    For $\varepsilon > 0$ sufficiently small, in a local neighbourhood of $(E,M)_{sBT(B)}$ and for $\alpha_1 \approx \alpha_{1,sTC}$, system \eqref{scalingC_reduced_again} can be transformed into the normal form:
    \begin{align} \label{BT_normal_sf}
    \begin{pmatrix}
        \frac{d\xi_1}{dt} \\
        \frac{d \xi_2}{dt} 
    \end{pmatrix} = \begin{pmatrix}
          \xi_2 \\  {B} \xi_1 \xi_2 +\varepsilon \left({\mu}_1  + {\mu}_2 \xi_1 + {A} \xi_1^2\right)
    \end{pmatrix}
    \end{align}
    where $\mu_1,\mu_2,A,$ and $B$ are smooth scalar functions of $(\alpha_1,\beta_1)$.
\end{lemma}
\begin{proof}
    {Following the approach in the proof of Lemma \ref{lemma:typeB}, recall that the coefficients $\tilde{\mu}_1, \tilde{\mu}_2, \tilde{A} \to 0$ as $(\alpha_1,\varepsilon) \to (0,0)$ (using the shifted $\alpha_1$ coordinate). Applying the parameter rescaling $\alpha_1 \to \varepsilon \alpha_1$ and retaining only the leading-order $\mathcal{O}(\varepsilon)$ terms yields the desired form.}
\end{proof}

{In classical smooth systems undergoing a transcritical Bogdanov-Takens (tBT) bifurcation, the standard saddle-node curve is replaced by a transcritical intersection. In biological models, such transcritical bifurcations are often enforced by an invariant physical boundary \cite{trans_BT}, which in our system corresponds to the tumour-free subspace $W^s_{sub} = \{M,C=0\}$. Furthermore, the vanishing of the normal form coefficients at $\varepsilon=0$ is a structural hallmark of a singular vector field, naturally leading to the conjectured \textit{singular} tBT (stBT) bifurcation. A complete topological unfolding of this stBT requires a full geometric parameter blow-up incorporating $\varepsilon$, $\alpha_1$, and $\beta_1$ simultaneously, a task we reserve for future work.}

\subsection{Improved singular bifurcation diagram}
\label{sec:improved_bif}

By synthesising these local results, we update our singular bifurcation diagram (Figure \ref{fig:improved}). The boundary previously identified as an unknown bifurcation is now confirmed as the parametric collision of the sTC and sBT(B) curves, organised by the stBT point at the origin where they collide spatially.

\subsubsection{Numerical validation and canard explosion}
Figure \ref{2Dbif} presents the two-parameter numerical continuation of the full 3D system \eqref{full_single_epsilon} for $\varepsilon = 0.01$. The boundaries mirror the singular unfolding derived in Figure \ref{fig:improved}, with the AH curves forming a closed loop enclosing the oscillatory regime (Region B).

\begin{figure}[ht]
\centering
\includegraphics[width=0.575\linewidth]{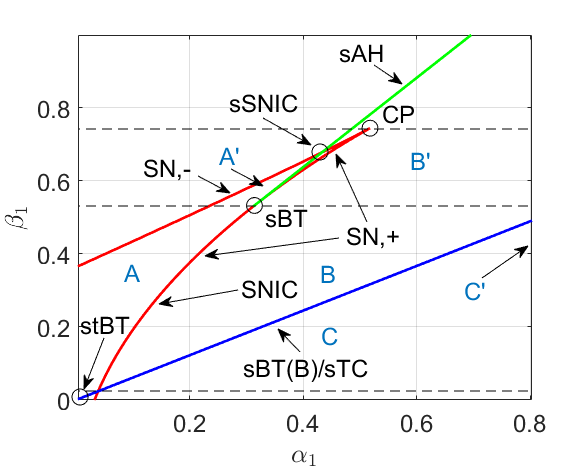}
\caption{An improved partial bifurcation diagram. The blue curve denotes the sTC and sBT(B), which spatially collide at the conjectured stBT point at the origin.}
\label{fig:improved}
\end{figure}

\begin{remark}[The generalised Hopf (Bautin) point]
The numerical presence of the generalised Hopf (GH) points in Figure \ref{2Dbif}d validates our analytical calculation of the first Lyapunov coefficient for the sAH(B) bifurcation, which yielded $\tilde{l}_1>0$ for sufficiently small $\beta_1>0$ (see Section SM4 in Supplementary Material II for details). However, as $(\alpha_1, \beta_1)$ move away from the pseudo-singularity, higher-order $\mathcal{O}(\varepsilon)$ corrections dominate, causing the first Lyapunov coefficient to change sign. 
\end{remark}

{Since the AH bifurcation is subcritical, its limit cycles must terminate before reaching the sTC boundary; that is, they must terminate at some saddle-node of periodic orbits (SNPO) at $\alpha_{1,SNPO} \in (\alpha_{1,sTC} - \mathcal{O}(\varepsilon), \alpha_{1,sTC})$. Homoclinic orbits are structurally forbidden from forming a connection with the saddle tumour-free equilibrium $\mathcal{EQ}_{tf}$ due to its invariant tumour-free subspace $W^s_{sub}$, which persists away from the singular limit (see Section \ref{sec:slow}). Thus, the system undergoes a complete canard explosion, where the small unstable AH limit cycles rapidly grow to become the large relaxation oscillations of Theorem \ref{main_theorem} (see Figure \ref{canard_cycles}).}
 
\subsection{Excitability in Region C}
\label{excitability2}
{
For sufficiently small $\varepsilon > 0$, Region C may be further subpartitioned into Regions C.I, C.II, and C.III based on the unfoldings of the bifurcations established in the blow-up of the origin, i.e., the sAH(B) (with its associated canard explosion and SNPO) and the sTC. Thus, we define: (i) Region C.I, where $\alpha_1\in(\alpha_{1,AH(B)},\alpha_{1,SNPO})$, which exhibits bistability; (ii) Region C.II, where $\alpha_1\in(\alpha_{1,SNPO},\alpha_{1,sTC})$, where $\tilde{\mathcal{EQ}}_{s}$ acts as the global attractor; and (iii) Region C.III, where $\alpha_1 >\alpha_{1,sTC}$, where $\tilde{\mathcal{EQ}}_{tf}$ is the global attractor. }
{Figure \ref{fig:regionC} shows the dynamics in these various subregions. Figure \ref{fig:regionC}a represents Region C.I, where large oscillations are still possible, but some solutions may be trapped inside the boundary enclosed by the unstable AH limit cycle, thus falling into the basin of attraction of $\tilde{\mathcal{EQ}}_s$. }

\begin{figure}[ht]
\centering
\begin{subfigure}{0.45\textwidth}
\centering
  \includegraphics[width=1\linewidth]{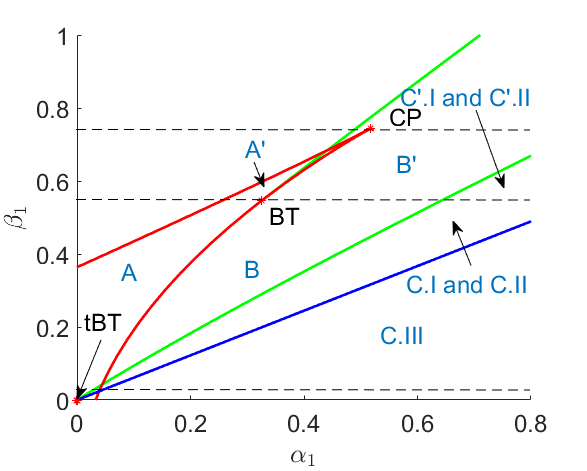}
  \caption{}
\end{subfigure}
\begin{subfigure}{0.45\textwidth}
\centering
  \includegraphics[width=1\linewidth]{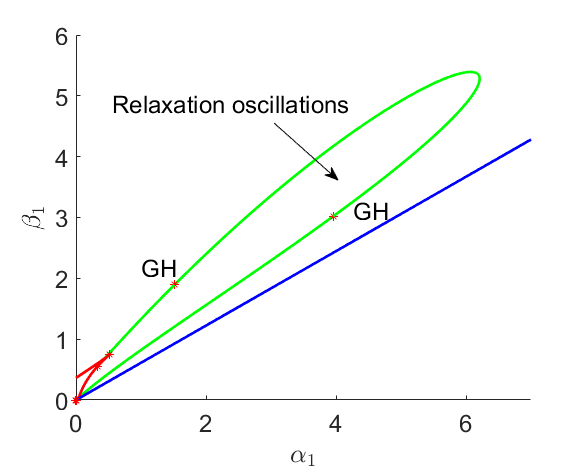}
  \caption{}
\end{subfigure}
\caption{{Partial bifurcation diagram of system \eqref{full_single_epsilon} for varying $(\alpha_1,\beta_1)$ with $\varepsilon = 0.01$. Bifurcation curves (computed via MatCont \cite{matcont}): transcritical (blue), Andronov-Hopf (green), and saddle-node (red). Points: Bogdanov-Takens (BT), transcritical BT (tBT), cusp (CP), and generalised Hopf (GH). (a) Labelled with the relevant parameter regions. (b) A zoom-out showing the closed AH loop enclosing the oscillatory regime.}}
\label{2Dbif}
\end{figure}

{Figure \ref{fig:regionC}b shows a case in Region C.II where the extension of $S_{1,1}^\varepsilon$ near the origin (the green curve) coincides with the unstable manifold of $\tilde{\mathcal{EQ}}_{tf}$, occurring at some $\alpha_{1,con} \in ( \alpha_{1,AH(B)},\alpha_{1,sTC})$. Here, it is clear that closed oscillations are no longer possible due to this connection. Instead, a transient large excursion occurs for solutions beginning to the left of the green curve, which then return to chart $K_5$ and settle on $\tilde{\mathcal{EQ}}_{s}$. We emphasise that $\alpha_{1,SNPO} < \alpha_{1,con}$ since $W_{5,sub}^s$ is an invariant set. Away from the singular limit, flow entering chart $K_5$ from $E\to \infty, M\to 0$ can never lie on $W_{5,sub}^s$; hence, the limit cycles cannot grow large enough to coincide with it, forcing them to terminate via an SNPO at an earlier parameter value $\alpha_1 \in (\alpha_{1,AH(B)}, \alpha_{1,con})$. Figure \ref{fig:regionC}a may also represent a case in Region C.II, where solutions beginning to the left of the green invariant manifold return to chart $K_5$ above it.}

{Recall that it is the possibility of these transient large excursions in Regions C.II and C.III that makes excitable behaviour possible. Furthermore, Region C' undergoes a similar partitioning, where one subregion is bistable (Region C'.I) and the others exhibit excitable behaviour (Regions C'.II and C'.III).}

\begin{figure}[ht]
\centering
\includegraphics[width=0.45\linewidth]{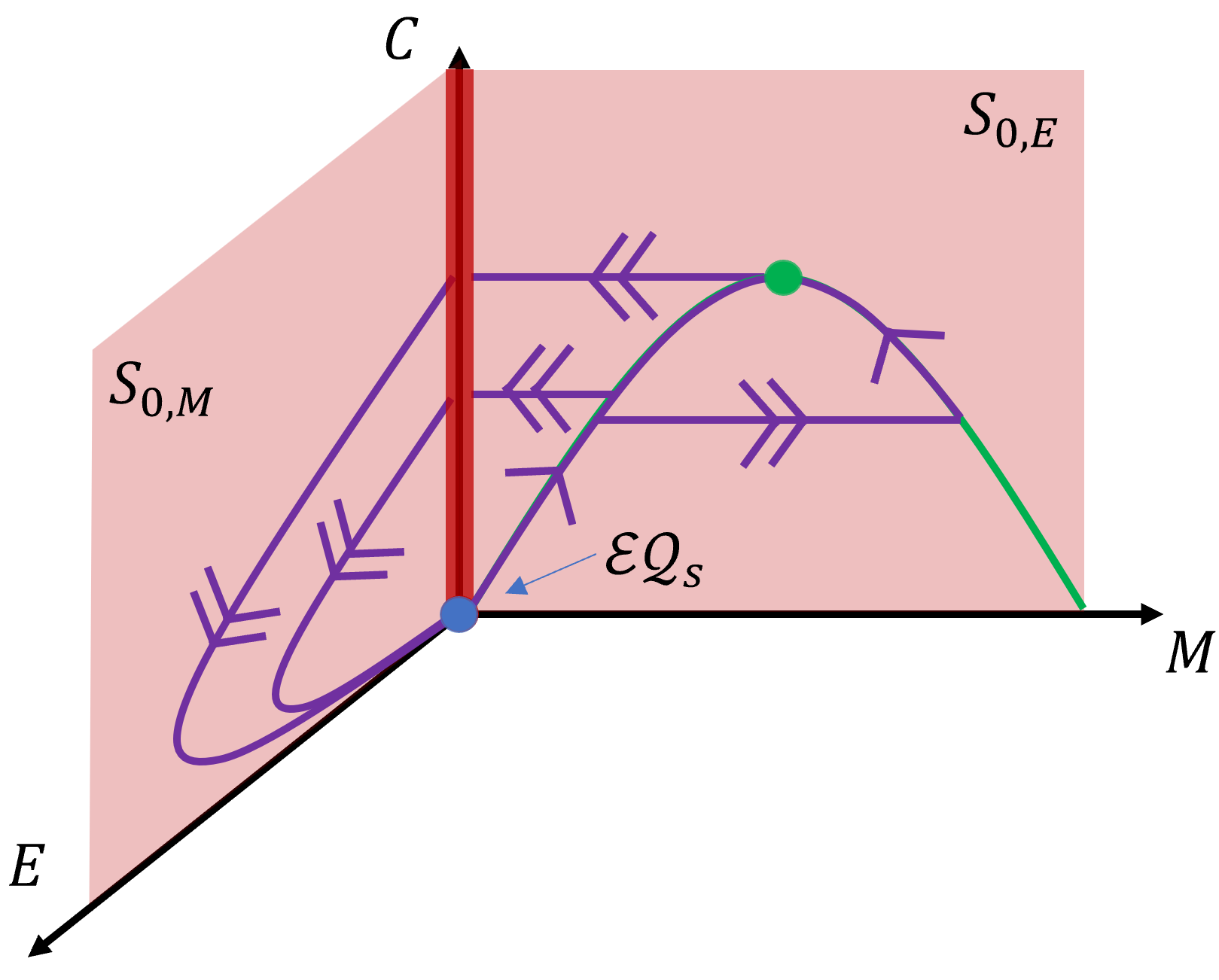}
  \caption{A complete family of singular canard cycles expanding into full relaxation oscillations, trapped away from the tumour-free equilibrium.}
\label{canard_cycles}
\end{figure}

{Figure \ref{fig:regionC}c also shows another case in Region C.II for $\alpha_1 \in (\alpha_{1,con},\alpha_{sTC})$ without the special connection between the green invariant manifold and $\tilde{\mathcal{EQ}}_{s}$. Transient large excursions are again possible if we are to the left of the green invariant manifold. Figure \ref{fig:regionC}d then represents in Region C.III, where the global attractor is now $\tilde{\mathcal{EQ}}_{tf}$ is the global attractor.}

\begin{figure}[ht]
\centering
\begin{subfigure}{0.425\textwidth}
\centering
 \includegraphics[width=0.95\linewidth]{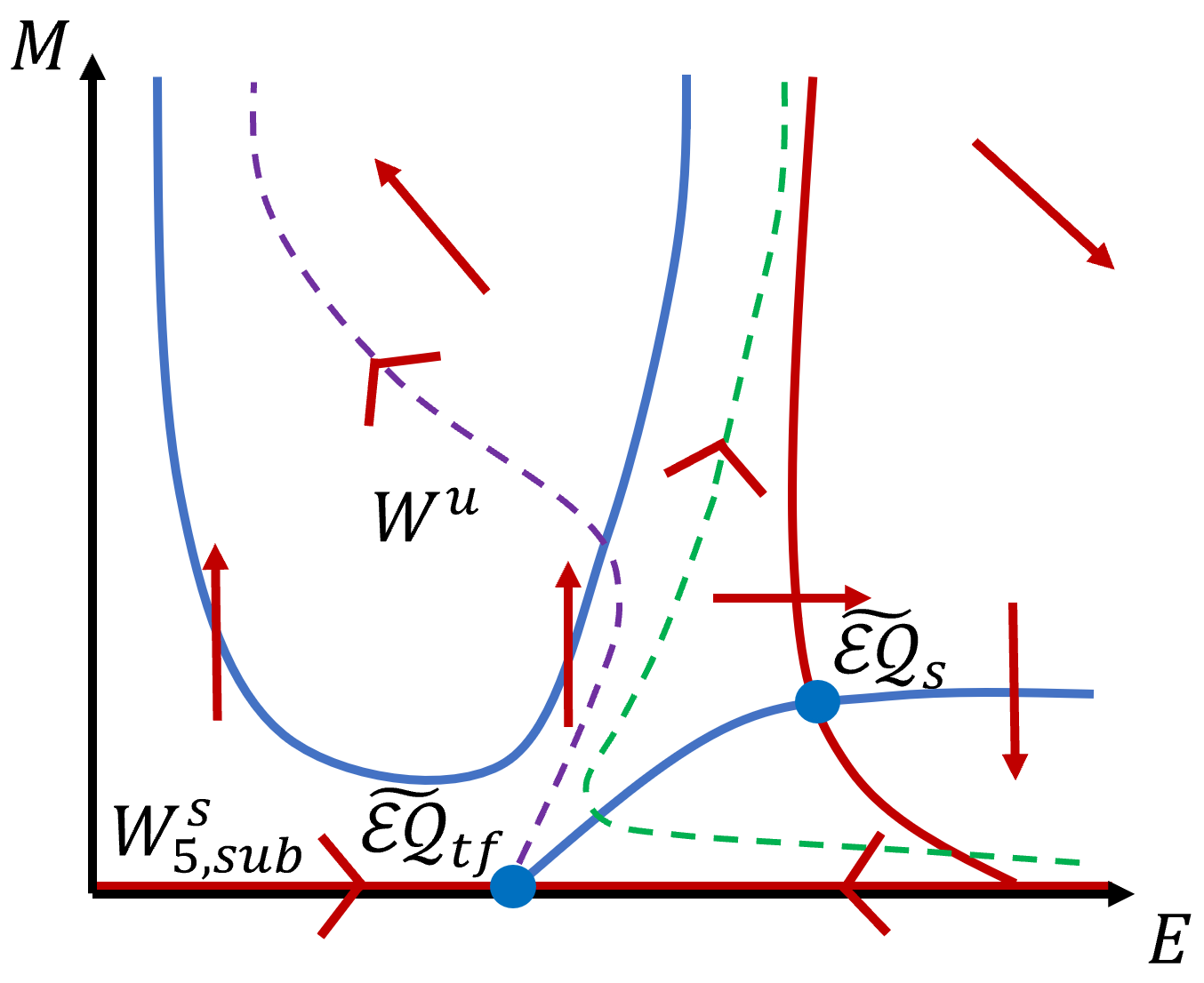}
  \caption{$\alpha_1 < \alpha_{1,con}$}
\end{subfigure}
\begin{subfigure}{0.425\textwidth}
\centering
  \includegraphics[width=0.90\linewidth]{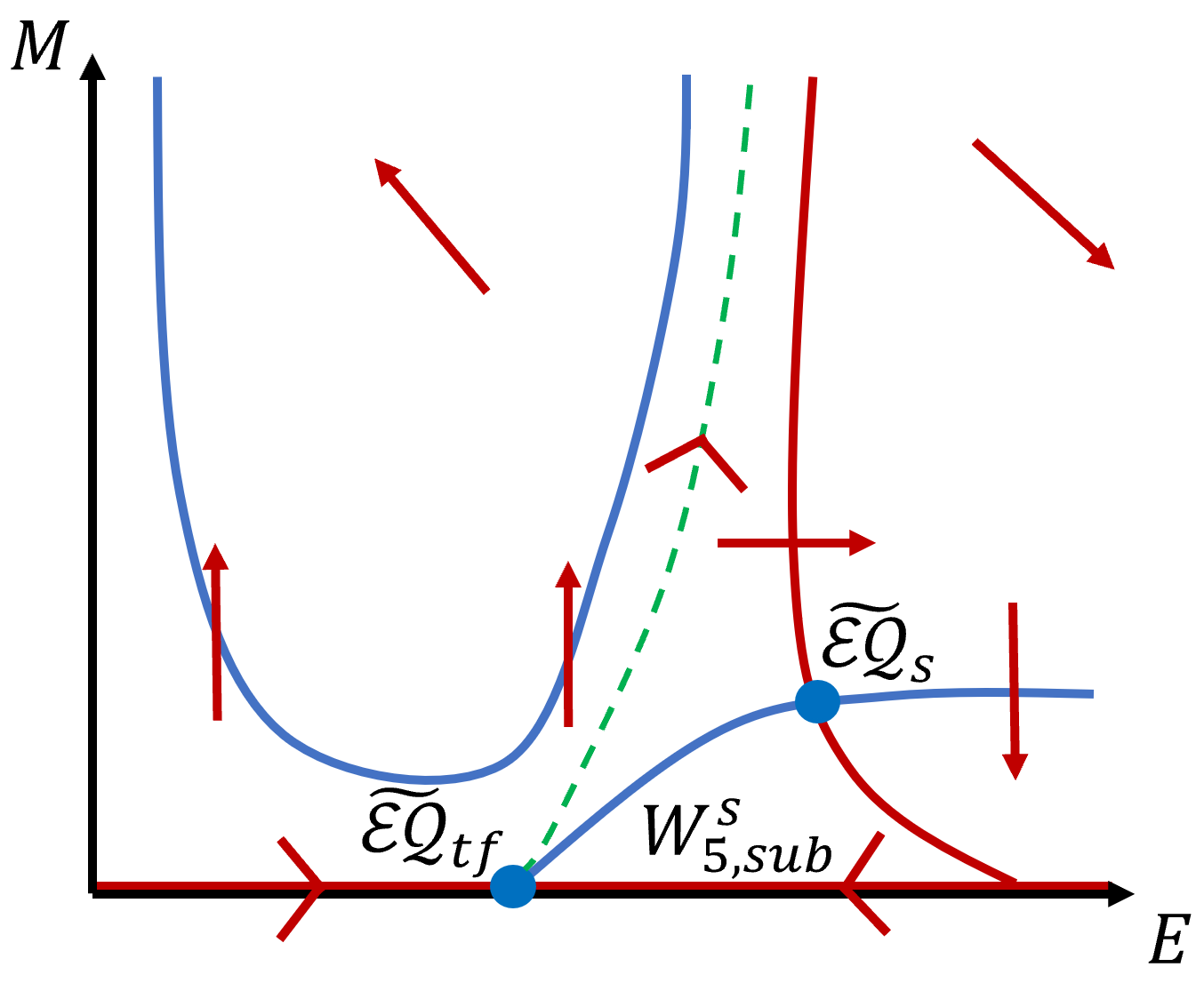}
  \caption{$\alpha_1 =  \alpha_{1,con}$}
\end{subfigure}
\begin{subfigure}{0.425\textwidth}
\centering
 \includegraphics[width=0.95\linewidth]{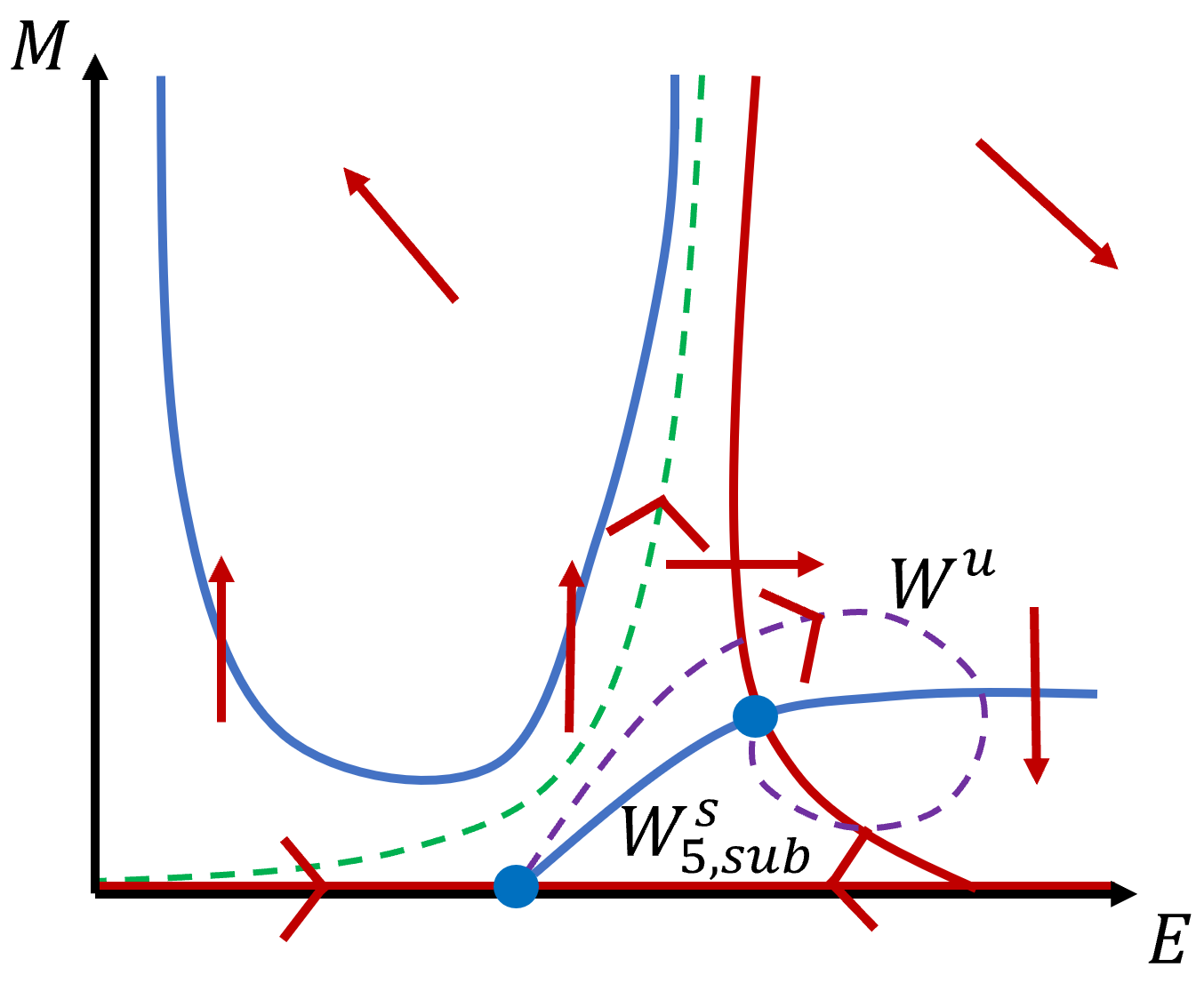}
  \caption{$\alpha_1 \in (\alpha_{1,con},\alpha_{1,sTC})$}
\end{subfigure}
\begin{subfigure}{0.425\textwidth}
\centering
  \includegraphics[width=0.95\linewidth]{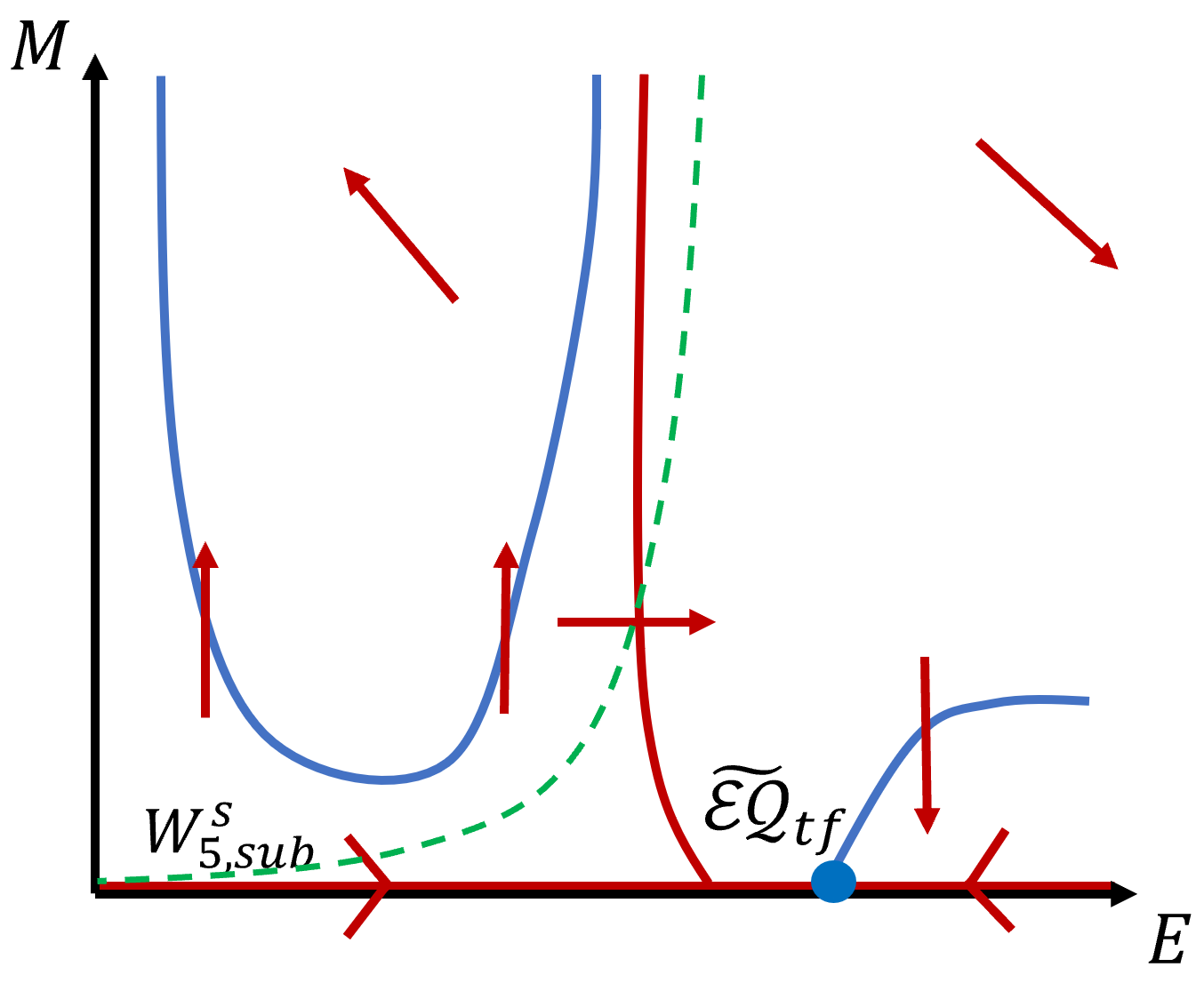}
  \caption{$\alpha_1 > \alpha_{1,sTC}$}
\end{subfigure}
\caption{Phase portrait sketches of system \eqref{scalingC_reduced_again} sub-partitioning Region C dynamics near the origin. The invariant curve (green) acts as a topological separatrix between trajectories that immediately decay and those that undergo transient canard-like excursions before settling on the tumour-free state.}
\label{fig:regionC}
\end{figure}

\section{Conclusion}
\label{conclusion}
In this paper, we presented a rigorous geometric analysis of the five-compartment tumour-immune model introduced by Kuznetsov et al. \cite{kuznetsov} and Osojnik et al. \cite{osojnik}. By uncovering the singular bifurcation structure, we partitioned the parameter space into distinct biological regimes, predicting whether a patient's tumour load will escape, oscillate, or settle into a dormant or eliminated steady state. 

We employed \textit{geometric singular perturbation theory (GSPT)} \cite{fenichel,kuehn2015,jones,wechselberger2020} and the \textit{blow-up method} \cite{dumortierroussarie1996,krupaszmolyan2001fold,szmolyanwechselberger} to formally prove the existence of relaxation oscillations and excitable large excursions. A defining mathematical feature of this system is its three-time-scale structure (Figure \ref{time_scales}), requiring the resolution of dynamics separated by up to three orders of magnitude (from $\mathcal{O}(1)$ to $\mathcal{O}(\varepsilon^3)$). To manage this, we utilised the \textit{parametrisation method} \cite{cabre20031,cabre20032,cabre2005,multiple} to systematically compute the invariant manifolds and their associated slow vector fields. 

This higher-order analysis uncovered a rich variety of singular bifurcations. We identified two singular Andronov-Hopf (sAH) bifurcations that trigger canard explosions. Most notably, at the origin, our analysis revealed a singular transcritical (sTC) bifurcation that collides parametrically with a nilpotent pseudo-singularity (associated with one of the sAH bifurcations) under the variation of $(\alpha_1,\varepsilon)$. Their spatial collision upon further variation of $\beta_1$ is conjectured to form a \emph{singular transcritical Bogdanov-Takens (stBT)} organising centre—a structure yet to be formally analysed in singular systems \cite{tBT,trans_BT,trans_BT2}. 

Furthermore, the parametrisation method proved indispensable for the blow-up analysis itself. We demonstrated that higher-order centre manifold corrections in the exit chart of the cylinder ($K_3$) and the entry chart of the sphere ($K_4$) were required to match the slow vector field on ${S}_{0,M}^\varepsilon$, a technical necessity that, to our knowledge, has not been explicitly addressed in previous blow-up applications \cite{kosiukszmolyan,kosiukszmolyan2,kosiukszmolyan3,process}. 

A comprehensive topological unfolding of the stBT organising centre requires a full geometric parameter blow-up and remains a highly anticipated topic for future research. Likewise, the numerical continuation of the associated saddle-node of periodic orbits (SNPO) and homoclinic bifurcations will provide a more complete partitioning of the parameter space away from the singular limit. 

{As noted in \cite{multiple}, `the parametrisation method may serve as a useful tool, as a more complete theory for the loss of normal hyperbolicity is developed'. Indeed, we have demonstrated that the parametrisation method serves as a vital complementary tool not only to classical Tikhonov-Fenichel theory but also directly aids the geometric blow-up method by uncovering necessary higher-order corrections to the local vector fields.}

\begin{figure}[ht]
\centering
\includegraphics[width=0.475\linewidth]{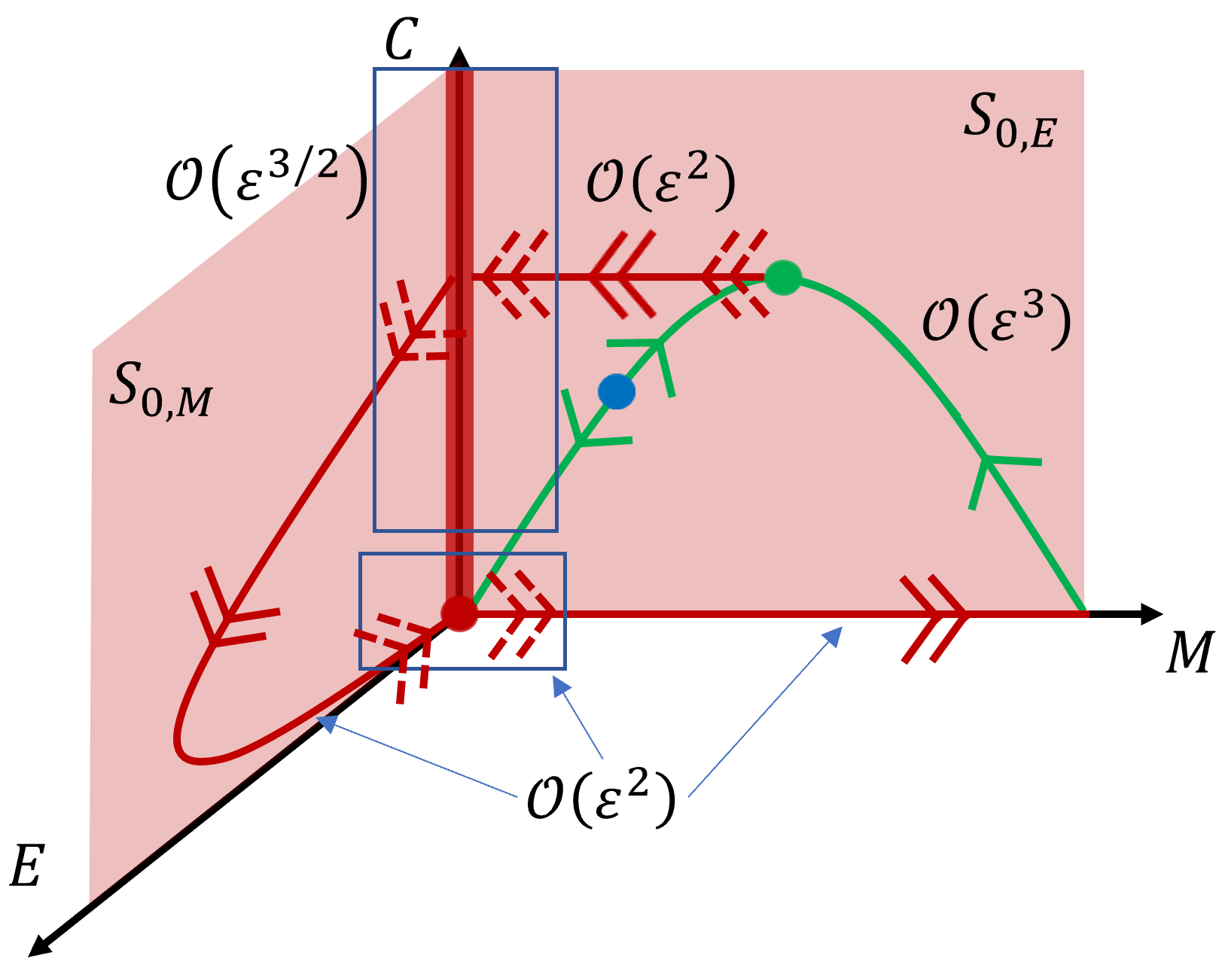}
  \caption{Schematic diagram indicating the distinct time scales of the flow on the critical manifolds ${S}_{0,E}^\varepsilon$ and ${S}_{0,M}^\varepsilon$, and near the non-hyperbolic $C$-axis and origin. The time scales are given with respect to $t$, the fast time of the layer problem of \eqref{full_single_epsilon}.}
\label{time_scales}
\end{figure}

\section*{Acknowledgments}
The first author thanks Vivien Kirk (University of Auckland, NZ) for travel support to the University of Auckland where part of this research was conducted. The first author also thanks John Bailie (University of Auckland, NZ) for discussions on bifurcation theory and the numerical continuation of the AH limit cycles to the relaxation oscillations through an SNPO in Figure \ref{1Dbif}.

\subsection*{Authorship and Contribution} All authors have made substantial intellectual contributions to the study conception, execution, and design of the work. All authors have read and approved the final manuscript. TEFL: formal analysis and investigation, writing, review and editing. MW: conceptualization, supervision, review and editing.

\subsection*{Conflict of Interest}
The authors declare there are no conflicts of interest.

\subsection*{Funding} TEFL acknowledges the support of an Australian Government Research Training Program (RTP) Scholarship. MW acknowledges the support through the Australian Research Council (ARC) Discovery Project grant DP220101817.

\bibliographystyle{siamplain}
\bibliography{references}

\end{document}

%% file: ex_shared.tex
\usepackage{lipsum}
\usepackage{amsfonts}
\usepackage{graphicx}
\usepackage{epstopdf}
\usepackage{algorithmic}
\ifpdf
  \DeclareGraphicsExtensions{.eps,.pdf,.png,.jpg}
\else
  \DeclareGraphicsExtensions{.eps}
\fi
\definecolor{Red}{rgb}{1,0.0.25,0.25}
\definecolor{Green}{rgb}{0.25,0.75,0.25}
\definecolor{Blue}{rgb}{0.1,0.5,1}

\newsiamremark{remark}{Remark}
\newsiamremark{hypothesis}{Hypothesis}
\crefname{hypothesis}{Hypothesis}{Hypotheses}
\newsiamthm{claim}{Claim}
\newsiamremark{fact}{Fact}
\crefname{fact}{Fact}{Facts}

\headers{Blow-Up of a Tumour Model}{T. E. F. Lapuz and M. Wechselberger}

\title{An Example Article\thanks{Submitted to the editors DATE.
\funding{This work was funded by the Fog Research Institute under contract no.~FRI-454.}}}

\author{Dianne Doe\thanks{Imagination Corp., Chicago, IL 
  (\email{ddoe@imag.com}, \url{http://www.imag.com/\string~ddoe/}).}
\and Paul T. Frank\thanks{Department of Applied Mathematics, Fictional University, Boise, ID 
  (\email{ptfrank@fictional.edu}, \email{jesmith@fictional.edu}).}
\and Jane E. Smith\footnotemark[3]}

\usepackage{amsopn}
